\documentclass[12pt,twoside]{amsart}
\input{includeNice}
\usepackage{mabliautoref}
\usepackage{amssymb}
\usepackage{amscd}
\usepackage[abbrev,alphabetic]{amsrefs}
\usepackage{hyperref}
\usepackage[a4paper,margin=1in]{geometry}  
\usepackage{soul}

\newenvironment{claimproof}[0]
{%
\paragraph{\it Proof of Claim.}%
}
{%
\hfill$\blacksquare$%
}

\newenvironment{s1proof}[0]
{%
\paragraph{\it Proof of Step 1.}%
}
{%
\hfill$\blacksquare$%
}

\newenvironment{s2proof}[0]
{%
\paragraph{\it Proof of Step 2.}%
}
{%
\hfill$\blacksquare$%
}

\newenvironment{s3proof}[0]
{%
\paragraph{\it Proof of Step 3.}%
}
{%
\hfill$\blacksquare$%
}
\newenvironment{s4proof}[0]
{%
\paragraph{\it Proof of Step 4.}%
}
{%
\hfill$\blacksquare$%
}

\usepackage{xypic}
\usepackage{tikz-cd}
\usepackage{caption}

\usepackage{url}
\RequirePackage{booktabs, multirow}
\RequirePackage{pgf}

\newcommand{\FB}[1]{\textcolor{red}{(FB: #1)}}

\newcommand{\cyan}{\color{black}}

\setlist[enumerate]{font=\normalfont}

\title[
Geometrically integral regular del Pezzo surfaces]
{Geometry and arithmetic of geometrically integral regular del Pezzo surfaces} 
\author[Fabio Bernasconi and Hiromu Tanaka]{Fabio Bernasconi and Hiromu Tanaka} 
\subjclass[2020]{}
\thanks{}

\address{Dipartimento di Matematica “Guido
Castelnuovo”, SAPIENZA Università di Roma, Piazzale Aldo Moro 5, I-00185,
Roma} 
\email{fabio.bernasconi@uniroma1.it}

\address{Department of Mathematics, Graduate School of Science, 
Kyoto University, Kyoto 606-8502, JAPAN}
\email{tanaka.hiromu.7z@kyoto-u.ac.jp}

\begin{document}
	
\begin{abstract}
We classify geometrically integral regular del Pezzo surfaces which are not geometrically normal over imperfect fields of positive characteristic. 
Based on this classification, we 
show that 
a three-dimensional terminal del Pezzo fibration onto a curve  over an algebraically closed field always admits a section. 
Moreover, we prove that the total space is rational if the base curve is rational and the anticanonical degree of a fibre is at least five. 
\end{abstract}
	
	\maketitle
	\tableofcontents

\section{Introduction}



We work over a field $k$. A regular del Pezzo surface $X$ 
(i.e., a regular projective surface such that $-K_X$ is ample) is one of the possible outcomes of the Minimal Model Program (for short, MMP) for regular projective surfaces.
When $k$ is algebraically closed, del Pezzo surfaces are completely classified: either $X$ is isomorphic to 
$\mathbb{P}^1_k \times \mathbb{P}^1_k$ or 
the blow-up of $\P^2_k$ at most at 8 points in general position \cite{Dol12}. 
The study of del Pezzo surfaces in the case when $k$ is perfect and their arithmetic properties has begun with the work of the italian school (e.g., \cite{Enr97, Seg51}) and the seminal article of Manin \cite{Man66}. Various problems concerning the geometry and arithmetic of del Pezzo surfaces remain active areas of research. Examples include the investigation of the unirationality of low-degree del Pezzo surfaces (see, for instance, \cite{STVA, FL16}), the exploration of the existence of rational points over various fields (e.g., \cite{CTM, ct20}), and the classification of their automorphism groups and twisted forms (see \cite{Shr20, SZ, Boi, Yas22, DM22, DM23}).

In this article, we study the geometry and the arithmetic of regular del Pezzo surfaces over arbitrary fields of positive characteristic, including the case of imperfect fields. 
This is motivated by the three-dimensional MMP: one of the possible outcomes of an MMP for a smooth threefold  over an algebraically closed field is a del Pezzo fibration $X \to B$, where $X$ is a $\mathbb{Q}$-factorial terminal threefold, 
$-K_X$ is ample over $B$, and $\dim(B)=1$.
In particular, the generic fibre $X_{K(B)}$ is a regular del Pezzo surface defined over an imperfect field. 
In characteristic $p>0$, generic smoothness 
does not hold in general and $X_{K(B)}$ can fail to be geometrically regular or even geometrically normal.

Nevertheless, various bounds on the pathological behaviours of $X_{K(B)}$ have been obtained in recent years \cite{BM23, BT22,  FS20, JW21, PW, Tan19}.
Our first theorem 
is to give a complete classification of geometrically integral regular del Pezzo surfaces which are not geometrically normal.

\begin{theorem}[\autoref{thm: class_dP}]\label{t-main-classify}
    Let $k$ be  a field of characteristic $p>0$. 
    Let $X$ be a regular del Pezzo surface over $k$. 
    Assume that $X$ is {geometrically integral and} not geometrically normal. 
    Then $p \in \{2, 3\}$ and $X$ belongs to one of the families in \autoref{table p=3} and \autoref{table p=2} {of \autoref{thm: class_dP}}.
\end{theorem}

It is important to classify the geometrically non-normal case, 
because the geometrically normal regular del Pezzo surfaces are 
automatically geometrically canonical, i.e., 
the base change $X \times_k \overline k$ is a del Pezzo surface with at worst canonical singularities, and these have been classified \cite[Section 8]{Dol12}. 
Thanks to description of the geometrically non-normal case (\autoref{t-main-classify}) 
and the Sarkisov program for regular surfaces developed in \cite{BFSZ}, 
we prove the following 
rationality criteria of regular del Pezzo surfaces, which extend the work of Enriques, Segre, Manin, Iskovskikh and others (\cite{Enr97, Man66, Isk79, SD72, Sko93, SB92, Isk96}) to the (possibly non-smooth) geometrically integral case.
{Recall that a surface $X$ over $k$ is {\em rational} if there exists a birational map $X \dashrightarrow \P^2_k$ over $k$.}

\begin{theorem}[\autoref{p deg7}, \autoref{t rationality}, \autoref{them: Enriques}] \label{thm: rationality_regular_dP}
	Let $k$ be a field of characteristic $p>0$.
	Let $X$ be a 
    regular del Pezzo surface over $k$ {which is geometrically integral}. 
Then the following hold. 
\begin{enumerate}
\item If $K_X^2 \geq 5$ and $X(k) \neq \emptyset$, then $X$ is rational;
\item If $K_X^2 =5$ or $K_X^2 =7$, then $X$ is rational. 
\end{enumerate}
\end{theorem}

{
\noindent
As a consequence of the above theorem, we establish the following rationality criterion {for Mori fibre spaces of dimension two}. 

\begin{theorem}[\autoref{thm: rat_min_MFS}]\label{intro: rat_min_MFS}
Let $k$ be a field of characteristic $p>0$. 
Let $X$ be a projective regular surface over $k$. 
Assume that $X$ is geometrically integral and $X$ admits a Mori fibre space structure $\pi \colon X \to B$. 
Then the following are equivalent. 
\begin{enumerate}
\item $X$ is rational. 
\item $K_X^2 \geq 5$ and  $X(k) \neq \emptyset$. 
\end{enumerate}
\end{theorem}
}

\noindent 
In the last section, we present sufficient conditions for the unirationality of regular del Pezzo surfaces of degree 4 (\autoref{thm: unirat_deg_4}) and 3  
(\autoref{thm: unirat_cubics}), extending work of Koll\'{a}r \cite{Kol02}. This is optimal due to the examples of Oguiso--Schr\"{o}er \cite{OS22}.

In \cite{Man66}, Manin showed  the existence of rational points for
a large class of smooth del Pezzo surfaces over perfect $C_1$-fields,   
and this has been later extended to all smooth del Pezzo surfaces (see \cite{Isk79} and \cite[Proposition 2.6]{CT86}). 
We extend their results by showing the $C_1$ conjecture for  regular del Pezzo surfaces which are not necessarily smooth. 
Note that this result is new also 
in the case of the function field of a curve.
(For a nice overview on the state of art of the $C_1$ conjecture, we refer to \cite{Esn23}.)

\begin{theorem}[\autoref{thm: existence_section}, cf. \autoref{r non geom int}(4)] \label{thm: C_1}
   Let $X$ be a regular del Pezzo surface over a $C_1$-field $k$  with $\dim_k H^0(X, \MO_X) =1$. 
   Then $X(k) \neq \emptyset$.
\end{theorem}

As immediate consequences of 
\autoref{thm: rationality_regular_dP} and 
\autoref{thm: C_1}, 
we deduce the following applications for three-dimensional del Pezzo fibrations.

\begin{theorem} \label{main dP fib}
Let $k$ be an algebraically closed field of characteristic $p>0$. 
Let $\pi : X \to B$ be a projective $k$-morphism such that 
$X$ is a terminal threefold over $k$, $B$ is a smooth curve over $k$, $-K_X$ is $\pi$-ample, 
and $\pi_*\MO_X = \MO_B$. 
Then the following hold. 
\begin{enumerate}
\item 
There exists a section of $\pi$. 
\item 
$X$ is rational if $B$ is a rational curve and 
$K_X^2 \cdot F \geq 5$ for a fibre $F$ of $\pi$ over a closed point of $B$. 
\end{enumerate}
\end{theorem}

As another application, we improve the results on the torsion index of numerically trivial line bundles on del Pezzo fibrations proven in \cite[Theorem 1.1]{BT22}.

\begin{corollary}[\autoref{ss rat pt}] \label{cor: torsion_index}
    Let $k$ be an algebraically closed field of characteristic $p>0$.
    Let $\pi : X \to B$ be a projective $k$-morphism such that $\pi_* \MO_X=\MO_B$, where
    $X$ is a $\mathbb{Q}$-factorial terminal threefold over $k$ and $B$ is a smooth curve over $k$.
    Suppose $\pi$ is a $K_X$-Mori fibre space, i.e., $\rho(X/B)=1$ and $-K_X$ is $\pi$-ample. If $L$ is a $\pi$-numerically trivial Cartier divisor on $X$, then $L \sim_{\pi} 0$.
\end{corollary}

\subsection*{Acknowledgements}
FB was supported by PZ00P2-21610 from the Swiss National Science Foundation. 
HT was supported by JSPS KAKENHI Grant number JP22H01112 and JP23K03028. 
We would like to thank Jérémy Blanc, Gebhard Martin, and Francesco Russo for useful discussions and references on the geometry of del Pezzo surfaces. 
 We  
thank the referees for reading the manuscript carefully and for suggesting several  improvements.

\section{Preliminaries}

\subsection{Notation}

\begin{enumerate}
    \item Throughout this article, $p$ denotes a prime number, 
    $k$ denotes a field of characteristic $p>0$, and 
    we work over $k$ (i.e., we work within the category of $k$-schemes) unless otherwise specified. 
    We denote by $\overline{k}$ (resp. $k^{\sep}$) the algebraic (resp. separable) closure of $k$.
    \item We say that $X$ is a {\em $k$-variety} (or simply a {\em variety}) if it is an integral separated scheme of finite type over $k$. 
    We say that $X$ is a {\em curve} (resp. a {\em surface}, a {\em threefold}) if it is of dimension one (resp. two, three). Given a variety $X$, $X^N$ denotes its normalisation. {Note that being a variety is not stable under base changes of base fields.}
    \item Given a field extension $k' / k$ and a  $k$-scheme $X$, 
    we set $X_{k'} := X \times_k k' := X \times_{\Spec(k)} \Spec(k')$. 
    \item For the definition of the singularities of the MMP, 
    we refer to \cite{kk-singbook}.
    {In particular, we say that a $k$-variety $X$ is {\em canonical} if $X$ is a normal variety over $k$ which has at worst canonical singularities.} 
    \item     
    A $k$-variety $X$ is {\em geometrically normal} (resp. {\em geometrically reduced}, {\em geometrically canonical}) if its base change $X_{\overline{k}}$ is normal (resp. reduced, canonical). 
    \item Given a $k$-variety $X$ with the structure morphism $\pi \colon X \to\Spec k$, we define the dualising complex $\omega_X^{\bullet} \coloneqq \pi^!\MO_{\Spec k}$ (see \stacksprojs{0A9Y}{0ATZ}). 
    The dualising sheaf is defined as the lowest non-zero cohomology sheaf of the dualising complex: 
    $\omega_X \coloneqq \mathcal{H}^{-\dim X}(\omega_X^{\bullet})$. For example, if $X$ is projective and smooth over $k$, then $\omega_X$ agrees with the usual sheaf of top differentials (\stacksproj{0AU3}).
    If $X$ is normal, then $\omega_X$ is a reflexive sheaf of rank 1 (\stacksproj{0AWE}), and we denote by $K_X$ a Weil divisor such that $\omega_X \simeq \MO_X(K_X)$.

    \item We say that a projective morphism $f \colon X \to Y$ is a {\em contraction} if $f_*\mathcal{O}_X = \mathcal{O}_Y$, i.e., 
    the induced homomorphism $\MO_Y \to f_*\MO_X$ is an isomorphism. {Given a projective birational morphism $f \colon X \to Y$ of varieties, $\Ex(f)$ denotes the {\em exceptional locus of $f$} that is defined as the smallest reduced closed subscheme $E$ of $X$ such that 
    the induced morphism $X \setminus E \to Y \setminus f(E)$ is an isomorphism.}
    \item Given a projective variety $X$, we denote by $\Pic(X)$ its Picard group and by $N^1(X)$ its numerical N\'eron-Severi group, i.e., 
    $N^1(X) := \Pic(X)/\equiv$  
    and $\equiv$ denotes the numerical equivalence for line bundles. 
    We set $N_1(X)_{\R} := \{ a_1 C_1 + \cdots + a_r C_r \,|\, r\geq 0, a_1, ...., a_r \in \R, C_1, ..., C_r 
    \text{ are curves on } X\} /\equiv$, 
    where $\equiv$ denotes the numerical equivalence for $1$-cycles. 
    \item Given a projective variety $X$, the Mori cone $\overline{\NE}(X)$ is the closure of the cone of effective $1$-cycles $\NE(X)$ in $N_1(X)_{\mathbb{R}}$.
    \item Given a Cartier divisor $D$ on a projective  variety $X$, the section ring associated to $D$ is the graded $k$-algebra $R(X, D)= \bigoplus_{m \geq 0} H^0(X, \MO_X(mD))$.
    \item We say that a $k$-variety $X$ is \emph{rational} (resp. \emph{unirational}) if there exists a birational map (resp. dominant rational map) of $k$-varieties 
    $\mathbb{P}^n_k \dashrightarrow X$. 
    \item 
 
   Given a projective normal  variety $X$ with $H^0(X, \mathcal{O}_X)=k$ and a $\mathbb{Q}$-Cartier $\Q$-divisor $D$, we denote its Iitaka dimension by $\kappa(D)$.  
   If $X$ is regular, then we set $\kappa(X) := \kappa(X, K_X)$, 
   which we call     the \emph{Kodaira dimension} 
   of $X$.
\end{enumerate}

\subsection{Regular del Pezzo surfaces}

A \emph{del Pezzo surface} (resp.\,\emph{weak del Pezzo surface}) $X$  
is a normal Gorenstein 
projective surface such that $H^0(X, \MO_X)=k$ and $-K_X$ is ample (resp. 
nef and big). 
In this section, we gather various results known in the literature and some  technical lemmas we need later.

As we are interested primarily in the case of imperfect fields, $X$ might not be 
smooth, nor geometrically normal.
However, we have the following restrictions on the possible bad behaviour:

\begin{proposition} \label{prop: geom_normality_degree}
Let $X$ be a geometrically integral regular del Pezzo surface.
\begin{enumerate}
\item If $X$ is geometrically normal, then it is geometrically canonical and $K_X^2 \leq 9$.
\item If $X$ is not geometrically normal, then either
\begin{enumerate}
    \item $p=3$ and $K_X^2 \in \left\{ 1, 3 \right\}$, or
    \item $p=2$ and $K_X^2 \in \left\{ 1, 2, 3, 4, 5, 6, 8 \right\}$.
\end{enumerate} 
\end{enumerate}
\end{proposition}

\begin{proof}
The  assertion (1) is proven in \cite{BT22}*{Theorem 3.3}, and the assertion (2) in \cite{Tan19}*{Theorem 4.6}.
\end{proof}

The following shows that del Pezzo surfaces with vanishing irregularity satisfy several vanishing theorems. 
{Note that a geometrically integral regular (or canonical) weak del Pezzo surface is not necessarily geometrically normal {nor have vanishing irregularity} \cite[Theorem in Introduction]{Sch07}.}
\begin{lemma} \label{lem: Kod_vanishing}
    Let $X$ be a geometrically integral canonical weak del Pezzo surface. 
    Assume $H^1(X, \MO_X)=0$. 
Then the following hold. 
\begin{enumerate}
\item If $N$ is a numerically trivial Cartier divisor, then $N \sim 0$. In particular, $\Pic(X)$ is a finitely generated free abelian group.
\item If $D$ is a nef and big Cartier divisor, then $h^0(X, \MO_X(D)) \geq 2$ and 
$H^1(X, \MO_X(-D))=0$. 
\item $H^i(X, \MO_X(-mK_X))=0$ for $i>0$ and $m \geq 0$.
\end{enumerate}
    

\end{lemma}

\begin{proof}
Taking the minimal resolution $\mu : Y \to X$, we may assume that $X$ is regular (this reduction step is assured, because 
canonical surface singularities are rational \cite[Proposition 2.28]{kk-singbook}, and hence 
we have 
$H^i(Y, \MO_Y(\mu^*L)) \simeq H^i(X, \MO_X(L))$ 
for every Cartier divisor $L$ on $X$ and $i \geq 0$).
Let $L$ be a nef Cartier divisor. By Serre duality, we have $h^2(X, \MO_X(L)) = h^0(X, \MO_X(K_X-L))=0$. 
It follows from the Riemann--Roch theorem that 
\[
h^0(X, \MO_X(L)) \geq h^0(X, \MO_X(L)) - h^1(X, \MO_X(L)) 
=\chi(X, \MO_X(L)) 
\]
\[
= \chi(X, \MO_X) + \frac{1}{2} L \cdot (L -K_X) = 1 +  \frac{1}{2} L \cdot (L -K_X) \geq 1. 
\]
The assertion (1) follows from $h^0(X, \MO_X(N)) \geq 1$. 


Let us show (2). 
We have 
\[
h^0(X, \MO_X(D)) \geq  1 +  \frac{1}{2} D \cdot (D -K_X) \geq 2. 
\]
It follows from \cite[Theorem 3.3]{Tan-FanoI} that 
$H^1(X, \MO_X(-D))=0$. 
Thus (2) holds. 
The assertion (3) follows from (2) and Serre duality. 
\qedhere 


\end{proof}

\begin{lem}\label{l va genus1}
Let $C$ be a projective Gorenstein curve with $h^1(C, \MO_C) =1$ {and let $D$ be a Cartier divisor on $C$}. 
Then the following hold. 
\begin{enumerate}
\item If $\deg D = 1$, 
then the base scheme $\Bs |D|$ is scheme-theoretically equal to a $k$-rational point $P$ and 
$R(C, D)$ is generated by $H^0(C, \MO_C(D)) \oplus H^0(C, \MO_C(2D)) \oplus H^0(C, \MO_C(3D))$ as a $k$-algebra. 
\item If $\deg D \geq 2$, then $|D|$ is base point free and 
$R(C, D)$ is generated by $H^0(C, \MO_C(D)) \oplus H^0(C, \MO_C(2D))$ as a $k$-algebra. 
\item If $\deg D \geq 3$, then $|D|$ is very ample 
and  
$R(C, D)$ is generated by $H^0(C, \MO_C(D))$ as a $k$-algebra. 
\item 
If $\deg D \geq 4$ and $C$ is smooth, then $C\subset \P^N_k$ is an intersection of quadrics in $\P^N_k$, 
where 
$N := h^0(C, \MO_C(D)) -1$ and $C$ is embedded into $\P^N_k$ by $\varphi_{|D|}$. 
\end{enumerate}
\end{lem}

\begin{proof}
The assertion (2) 
follows from \cite[Theorem 4.4 and Lemma 4.5]{Tan-elliptic}. 
The assertion (3) holds by \cite[Theorem 4.9]{Tan-elliptic}.
The assertion (4) follows from \cite[Theorem 10, page 80]{Mum70}.

Let us show (1). 
By the Riemann-Roch theorem, we get $h^0(C, \MO_C(D)) = \deg D = 1$, 
and hence $P:=\Bs |D|$  is a $k$-rational point. 
It follows from (2) that $|2D|$ is base point free. 
For $n \geq 3$, we get 
\[
H^1(C, \MO_C(nD) \otimes \MO_C(-2D)) = 0, 
\]
and hence $\MO_C(nD)$ is $0$-regular with respect to 
the  globally generated ample invertible sheaf $\MO_C(2D)$. 
Then the induced $k$-linear map 
\[
H^0(C, \MO_C(nD)) \otimes_k H^0(C, \MO_C(2D)) \to H^0(C, \MO_C((n+2)D))
\]
is surjective for $n \geq 3$ (cf. \cite[Lemma 5.1]{FAG}). 
Therefore, $R(C, D)$ is generated by $\bigoplus_{1 \leq j \leq 4}H^0(C, \MO_C(jD))$. 
It is easy to check that 
\[
H^0(C, \MO_C(2D)) \otimes H^0(C, \MO_C(2D)) 
\oplus 
H^0(C, \MO_C(D)) \otimes H^0(C, \MO_C(3D)) 
\to H^0(C, \MO_C(4D))
\]
is surjective. 
Thus (1) holds. 
\end{proof}

\begin{prop}\label{p va deg3}
Let $X$ be a geometrically integral 
canonical del Pezzo surface with $H^1(X, \MO_X)=0$. 
Assume that one of {\rm (a)} and {\rm (b)} holds.
\begin{enumerate}
\item[(a)]
There exists a prime divisor $C$ satisfying $-K_X \sim C$. 
\item[(b)] 
 $|-K_X|$ is base point free.
\end{enumerate}
Then the following hold. 
\begin{enumerate}
\item 
If $K_X^2 =1$, then $\Base |-K_X|=P$ is a $k$-rational point 
and $R(X, -K_X)$ is generated by $H^0(X, \MO_X(-K_X)) \oplus H^0(X, \MO_X(-2K_X)) \oplus H^0(X, \MO_X(-3K_X))$ as a $k$-algebra. 
\item 
If $K_X^2 =2$, then $|-K_X|$ is base point free and 
$R(X, -K_X)$ is generated by $H^0(X, \MO_X(-K_X)) \oplus H^0(X, \MO_X(-2K_X))$.  
\item 
If $K_X^2 \geq 3$, then $|-K_X|$ is very ample and  
$R(X, -K_X)$ is generated by $H^0(X, \MO_X(-K_X))$ as a $k$-algebra. 
\item 
If $X$ is geometrically normal and $K_X^2 \geq 4$, then $X$ is an intersection of quadrics. 
\end{enumerate}
\end{prop}

\begin{proof}
Assume that (a) holds. 
By \autoref{lem: Kod_vanishing}, 
we have the following  surjection 
\[
H^0(X, \MO_X(-mK_X)) \to H^0(C,\MO_C(-mK_X|_C))\to H^1(X, \MO_X((-m+1)K_X))=0. 
\]
Then the assertions 
follow from  the corresponding result for $C$ 
(\autoref{l va genus1}, \cite[Lemma 2.10]{Isk77} cf. \cite[Lemma 7.6]{Tan-FanoI}). 
{\cyan This completes the proof for the case when (a) holds.} 

{\cyan Assume (b). 
By \cite[Remark 3.2, Theorem 4.9(12), and Theorem 4.12]{Tan24}, 
there exists a purely transcendental extension $\kappa/k$ of finite degree such that 
(a) holds for the base change $X_{\kappa} = X \times_k \kappa$.  
Note, for example, that 
\[
\mu : H^0(X, \MO_X(-K_X)) \otimes_k H^0(X, \MO_X(-K_X)) \to 
H^0(X, \MO_X(-2K_X))
\]
is surjective if and only if $\mu \otimes_k \kappa$ is surjective. 
Then we are done as we have already settled the case when (a) holds. 
This completes the proof for the case when (b) holds.}
\end{proof}

\begin{prop}\label{p hypersurf}
Let $X$ be a geometrically integral canonical 
del Pezzo surface  with $H^1(X, \MO_X)=0$. 
Assume that one of {\rm (a)} and {\rm (b)} holds.
\begin{enumerate}
\item[(a)]
There exists a prime divisor $C$ satisfying $-K_X \sim C$. 
\item[(b)] 
 $|-K_X|$ is base point free.
\end{enumerate}
Then the following hold. 
\begin{enumerate}
\item If $K_X^2=1$, then 
$X$ is isomorphic to a hypersurface in $\P(1, 1, 2, 3)$ of degree $6$. 
\item If $K_X^2 =2$, 
then  $X$ is isomorphic to a hypersurface in $\P(1, 1, 1, 2)$ of degree $4$. 
\item If $K_X^2=3$, then  $X$ is isomorphic to a cubic hypersurface in $\P^3_k$. 
\item If $K_X^2 =4$, then $X$ is isomorphic to a complete intersection of two quadric hypersurfaces in $\P^4_k$. 
\end{enumerate}
\end{prop}

\begin{proof}
Using \autoref{p va deg3}, the assertions (1)-(3) hold by applying 
the same argument as in \cite[Theorem III.3.5]{Kol96}. 
Let us show (4). 
Assume $K_X^2=4$. By the Riemann-Roch theorem, we get 
\[
h^0(X, \MO_X(-K_X)) = 5\qquad\text{and}\qquad 
h^0(X, \MO_X(-2K_X))= 13. 
\]
We have $h^0(\P^4_k, \MO_{\P^4_k}(2))= \binom{6}{2}=15$. 
By the exact sequence 
\[
0 \to H^0(\P^4_k, I_X \otimes \MO_{\P^4_k}(2)) \to 
H^0(\P^4_k, \MO_{\P^4_k}(2)) \to 
H^0(X, \MO_X(-2K_X)),  
\]
there are two quadric hypersurfaces $Q$ and $Q'$ such that $Q \neq Q'$ and $X \subset Q \cap Q'$. 
If the quadric $Q$ is not an integral scheme, then $X$ would be contained in a hyperplane of $\mathbb{P}^4_k$, which contradicts the isomorphism: 
\[
H^0(\P^4_k, \MO_{\P^4_k}(1)) \xrightarrow{\simeq} H^0(X, \MO_{X}(-K_X)). 
\]
Therefore, both $Q$ and $Q'$ are prime Cartier divisors on $\P^4_k$. 
By $\deg X =4 = \deg (Q \cap Q')$, we get $X = Q \cap Q'$. 
\end{proof}


\begin{lem}\label{l rho=1}
Let $X$ be a geometrically integral regular del Pezzo surface 
{with $H^1(X, \MO_X)=0$ and $\Pic X = \Z K_X$}. 
Take a $k$-rational point $P$ on $X$ and 
let $\sigma: Y \to X$ be the blowup at $P$. 
Then the following hold. 
\begin{enumerate}
\item If $K_X^2 \geq 3$, then $|-K_Y|$ is base point free and $-K_Y$ is big. 
\item If 
$K_X^2 \geq 4$, 
then $-K_Y$ is ample. 
\end{enumerate}
\end{lem}


\begin{proof}
Let us show (1). 
{The Riemann-Roch theorem implies $H^0(X, \MO_X(-K_X)) \neq 0$. 
This, together with $\Pic X =\Z K_X$, enables us to find a prime divisor $C$ satisfying $-K_X \sim C$. 
Hence $|-K_X|$ is very ample (\autoref{p va deg3}). 
Then} $\sigma$ coincides with the resolution of indeterminacies of the linear system consisting of the hyperplane sections containing $P$. 
Therefore, 
\[
-K_Y = -\sigma^*K_X -E
\]
is base point free for $E:=\Ex(\sigma)$. 
By $K_Y^2 =K_X^2 -1 \geq 2>0$, $-K_Y$ is big. 
Thus (1) holds.

Let us show (2). 
Suppose that $-K_Y$ is not ample. Let us derive a contradiction. 
As $-K_Y$ is nef and big, the extremal ray not corresponding to $\sigma$ 
induces the contraction to its anticanonical model: 
\[
\tau : Y \to Z. 
\]
By $\rho(Y)= \rho(X) +1 = 2$, 
$C := \Ex(\tau)$ is a prime divisor. 
It holds that 
\[
C \equiv -\alpha K_Y -\beta E 
\]
for some $\alpha, \beta \in \Z$, 
because  $\Pic Y = \sigma^*\Pic X \oplus \Z E \equiv \Z K_Y \oplus \Z E$. 


For $d_Y := K_Y^2 =K_X^2 -1$, we get  
\[
0 = C \cdot (-K_Y)  = \alpha (-K_Y)^2 -\beta (-K_Y) \cdot E  = d_Y\alpha -\beta. 
\]
Hence 
\[
C \equiv \alpha (-K_Y -d_YE). 
\]
By $K_Y \cdot C =0$ and $C^2 <0$, 
$C$ is a Gorenstein curve of genus $0$, which is a conic. 
In particular, $\deg_{k_C} K_C =-2$ for $k_C :=H^0(C, \MO_C)$. 
We have 
\[
d_Y\Z \ni \deg_{k_C} \MO_Y(-d_Y\alpha E)|_C=
\deg_{k_C} \MO_Y(\alpha (-K_Y -d_YE))|_C
\]
\[
=\deg_{k_C} \MO_Y(C)|_C = \deg_{k_C} \MO_Y(K_Y+C)|_C = \deg_{k_C} K_C = -2,
\]
which contradicts $d_Y =K_Y^2 =K_X^2-1 \geq 3$.  
\qedhere

\end{proof}




\section{Classification of geometrically integral {regular} del Pezzo surfaces} 

In this section, we classify geometrically integral regular del Pezzo surfaces 
which are not geometrically normal (\autoref{thm: class_dP}). 

\subsection{Primitive case}

\begin{dfn}\label{d primitive}
Let $X$ be a geometrically integral regular del Pezzo surface. 
\begin{enumerate}
\item We say that $X$ is {\em imprimitive} 
if there exists a birational morphism $f: X \to Y$ 
to a regular projective surface $Y$ such that $\rho(Y) = \rho(X)-1$. 
It is well known that 
$Y$ is a geometrically integral regular del Pezzo surface (e.g., \cite[Lemma 2.14]{BM23}).
\item We say that $X$ is {\em primitive} if 
$X$ is not imprimitive. 
\end{enumerate}
\end{dfn}

Let $X$ be a primitive geometrically integral regular del Pezzo surface. Then one of the following holds. 
\begin{itemize}
    \item $\rho(X)=1$ (\autoref{p rho 1}). 
    \item $\rho(X)=2$ and both extremal rays of $\NE(X)$ induce  conic bundle structures  (\autoref{p 2 conic bdls}). 
\end{itemize}
In what follows, we treat the case when $H^1(X, \MO_X)=0$ and $X$ is not geometrically normal, as the case when 
$H^1(X, \MO_X) \neq 0$ is treated in \cite[Proposition 4.11]{BM23}. 





\begin{prop}[primitive, $\rho(X)=1$]\label{p rho 1}\label{l index 1}
Assume that $k$ is separably closed. 
Let $X$ be a geometrically integral regular del Pezzo surface. 
Assume that $\rho(X)=1$, $X$ is not geometrically normal, and $H^1(X, \MO_X)=0$. 
Then the following hold.
\begin{enumerate}
    \item [(a)] $\Pic X = \Z K_X$. 
    \item [(b)] If $C$ is an effective Cartier divisor satisfying $C \sim -K_X$, 
    then $C$ is a prime divisor. 
\end{enumerate}
Moreover, one of the following holds. 
\begin{enumerate}
\item $p \in \{2, 3\}$, $K_X^2 =1$, and $(X_{\overline k})^N \simeq \P^2$. 
\item $p =2$, $K_X^2=2$, and $(X_{\overline k})^N \simeq \P(1, 1, 2)$. 
\item $p=3$, $K_X^2 =3$, and $(X_{\overline k})^N \simeq \P(1, 1, 3)$. 
\item $p=2$, $K_X^2 =4$, and 
 either $(X_{\overline k})^N \simeq \P^2$ or $(X_{\overline k})^N \simeq \P(1, 1, 4)$.
\end{enumerate}
\end{prop}

\begin{proof}
Since $k$ is separably closed, we get $\rho((X_{\overline k})^N)  = \rho(X)=1$ 
\cite[Lemma  2.2, Proposition 2.4]{Tan18b}. 
Then one of (1)-(4) holds by \cite[Theorem 4.6]{Tan19}. 
In particular, $1 \leq K_X^2 \leq 4$.

Since (b) follows from (a), it is enough to show (a). 
By $\rho(X)=1$ {and $H^1(X, \MO_X)=0$}, 
we have $\Pic X = \Z H$ for some ample Cartier divisor $H$ 
(\autoref{lem: Kod_vanishing}). 
In particular, we get $-K_X \sim rH$ for some integer $r >0$. 
It suffices to show $r=1$. 
Suppose $r \geq 2$. It is enough to derive a contradiction. 
By $r^2 H^2 =(rH)^2 = K_X^2 \leq 4$, 
it holds that  $(r, H^2)=(2, 1)$. 
By the Riemann-Roch theorem, we get the following contradiction: 
\[
\Z \ni 
\chi(X, \MO_X(H)) = \chi(X, \MO_X) + 
\frac{1}{2}H \cdot (H-K_X) 
\]
\[
= 
\chi(X, \MO_X) + \frac{1}{2}H \cdot (H +2H)  = \chi(X, \MO_X) +\frac{3}{2} \not\in \Z. 
\]
Thus (a) holds.
\end{proof}

\begin{prop}[primitive, $\rho(X)=2$]\label{p 2 conic bdls}
Assume that $k$ is separably closed.   
Let $X$ be a geometrically integral regular del Pezzo surface. 
Assume that $X$ is not geometrically normal, $H^1(X, \MO_X)=0$, $\rho(X) =2$, and each of the extremal rays induces a conic bundle structure. 
Then {$p=2$ and one of the following holds. 
\begin{enumerate}
\item $K_X^2 =2$ and 
$D$ is a prime divisor for every effective Cartier divisor $D$ satisfying 
$-K_X \sim D$. 
\item $K_X^2=4$ and there exists a double cover 
\[
\pi : X \to \P^1_k \times_k \P^1_k
\]
such that, for each $i \in \{1, 2\}$, 
the composite morphism $\pi_i : X \xrightarrow{\pi} \P^1_k \times_k \P^1_k \xrightarrow{{\rm pr}_i} \P^1_k$ is a contraction of an extremal ray of $\NE(X)$. 
Moreover, $-K_X \sim F_1 + F_2$, 
where $F_i$ is a fibre of $\pi_i$ over a $k$-rational point of $\P^1_k$.
\end{enumerate}
}
\end{prop}


\begin{proof}
Let $\pi_1 : X \to B_1$ and $\pi_2 : X \to B_2$ be 
the contractions of the two extremal rays of $\NE(X)$.

\setcounter{step}{0}

\begin{step}\label{s1 2 conic bdls}
It holds that $B_i = \P^1_k$ for each $i \in \{1, 2\}$. 
\end{step}


\begin{s1proof}
Fix $i \in \{1, 2\}$. 
    Since $X$ is geometrically integral, so is $B_i$. 
By the Leray spectral sequence, we have an injection $H^1(B_i, \MO_{B_i}) \hookrightarrow H^1(X, \MO_X)=0$, 
so we deduce $H^1(B_i, \MO_{B_i})=0$. 
As 
$k$ is separably closed and 
$B_i$ is geometrically integral, 
$B_i$ has a $k$-rational point, which implies $B_i \simeq \P^1_k$. 
 This completes the proof of Step \ref{s1 2 conic bdls}. 
\end{s1proof}

\medskip

For $i \in \{1, 2\}$, we fix a $k$-rational point $Q_i \in B_i = \P^1_k$ 
and set $F_i := \pi_i^*Q_i$ to be its fibre. 

\begin{step}\label{s2 2 conic bdls}
For $a := F_1 \cdot F_2 \in \Z_{>0}$, the following hold. 
\begin{enumerate}
\item $-aK_X \sim 2F_1 + 2F_2$
\item $a K_X^2 = 8$. 
\item $p=2$ and $K_X^2 \in \{2, 4\}$. 
\end{enumerate}
\end{step}

\begin{s2proof}
Let us show (1). 
By $\NE(X) = \R_{\geq 0}[F_1] + \R_{\geq 0}[F_2]$, 
we can write 
\[
-aK_X \equiv b F_1 + c F_2
\]
for some $b, c \in \R_{\geq 0}$. 
Taking the intersection with $F_1$, we obtain 
\[
2a = -aK_X \cdot F_1 = (b F_1 + c F_2) \cdot F_1 = c F_2 \cdot F_1 = ac. 
\]
Hence $c=2$. By symmetry, we get $b=2$. 
Thus $-aK_X \equiv 2F_1 +2F_2$. 
Then (1) holds by \autoref{lem: Kod_vanishing}. 

The assertion (2) follows from 
\[
a^2 K_X^2 = (-aK_X)^2 \overset{{\rm (1)}}{=} (2F_1 + 2F_2)^2 =8F_1 \cdot F_2 = 8a. 
\]
By (2), we get $K_X^2 \in \{1,2,4,8\}$. 
Then the assertion (3) holds by the list of \cite[Theorem 4.6]{Tan19}.  
{Here the case $K_X^2 =1$ (resp. $K_X^2 =8$) is excluded 
by $\rho(X)=2$ (resp. the assumption that $X$ has two Mori fibre space structures to curves).}  
 This completes the proof of Step \ref{s2 2 conic bdls}. 
\end{s2proof}

\begin{step}\label{s3 2 conic bdls}
If $K_X^2=2$ and $D$ is an effective Cartier divisor satisfying $-K_X \sim D$, then 
$D$ is a prime divisor. 
\end{step}

\begin{s3proof}
There are three types of curves $\Gamma$ on $X$: 
\begin{enumerate}
    \item[(i)] $\pi_1(\Gamma)$ is a point. 
    \item[(ii)] $\pi_2(\Gamma)$ is a point. 
    \item[(iii)] $\Gamma$ satisfies none of (i) nor (ii), i.e., 
    $\Gamma$ is ample. 
\end{enumerate}

Suppose that $D$ contains no curve of type (iii). 
Then we get $D \sim b_1 F_1 + b_2 F_2$ for some $b_i \in \Z_{>0}$, which leads to the following contradiction: 
\[
2 = K_X^2 = D^2  = (b_1 F_1 + b_2 F_2)^2 =2b_1b_2 F_1 \cdot F_2 = 8b_1b_2 \in 8 \Z.  
\]
Therefore, $D$ contains an ample prime divisor $C$. 

We can write $D = C+D'$ for some effective divisor $D'$. 
It suffices to show $D'=0$. 
Since the extremal rays of $\NE(X)$ are spanned by nef curves,  
every effective divisor on $X$ is nef, and thus we get  
\[
2 =K_X^2 = D^2 = (C+D')^2 = C^2 +2 C \cdot D' + D'^2 \geq 
C^2 +2 C \cdot D'. 
\]
If $D' \neq 0$, then we would get the following contradiction: 
\[
2 \geq C^2 +2C \cdot D' \geq 1 +2 \cdot 1 =3. 
\]
Therefore, $D$ is a prime divisor. 
 This completes the proof of Step \ref{s3 2 conic bdls}. 
\end{s3proof}

\begin{step}\label{s4 2 conic bdls}
If $K_X^2=4$, then  there exists a double cover 
\[
\pi : X \to \P^1_k \times_k \P^1_k
\]
such that, for each $i \in \{1, 2\}$, 
the composite morphism $\pi_i : X \xrightarrow{\pi} \P^1_k \times_k \P^1_k \xrightarrow{{\rm pr}_i} \P^1_k$ is a contraction of an extremal ray of $\NE(X)$. 
Moreover, $-K_X \sim F_1 + F_2$. 
\end{step}

\begin{s4proof}
Assume $K_X^2 = 4$. 
In this case, $a = {F_1} \cdot {F_2} = 2$ and 
$-K_X \sim {F_1} + {F_2}$ (\autoref{lem: Kod_vanishing}). 
Let 
\[
\pi := \pi_1 \times \pi_2 : X \to \P^1_k \times_k \P^1_k
\]
be the induced morphism. 
Since the pullback $\pi^*\MO_{\P^1 \times \P^1}(1, 1) = 
{F_1} + {F_2} = -K_X$ is ample, 
$\pi : X \to \P^1_k \times_k \P^1_k$ is a finite surjective morphism. 
We have that 
\[
4 = (-K_X)^2 = (\pi^*\MO_{\P^1 \times \P^1}(1, 1))^2 
=
(\deg \pi)(\MO_{\P^1 \times \P^1}(1, 1))^2  = 2 \deg \pi. 
\]
Hence $\pi$ is a double cover. 
The remaining assertion follows from the construction. 
 This completes the proof of Step \ref{s4 2 conic bdls}. 
\end{s4proof}

\medskip

Step \ref{s1 2 conic bdls}-Step \ref{s4 2 conic bdls} complete the proof of \autoref{p 2 conic bdls}. 
\end{proof}

\subsection{Imprimitive case}

We now deal with geometrically non-normal regular del Pezzo which are imprimitive.

\begin{prop}[imprimitive, $Y$ geometrically non-normal]\label{p imprim geom-non}
Assume that $k$ is separably closed. 
Let $f : X \to Y$ be a birational morphism of 
geometrically integral regular del Pezzo surfaces 
such that $\rho(X) = \rho(Y)+1$. 
Assume that $H^1(X, \MO_X)=0$ and $Y$ is not geometrically normal (and hence neither is $X$). 
Then the following hold. 
\begin{enumerate}
\item $p=2$. 
\item $\rho(X)=2$, $K_X^2 =3$, and $K_Y^2=4$.
\item The blowup centre of $f$ is a $k$-rational point. 
\item {$X$ is isomorphic to a cubic hypersurface in $\P^3_k$.} 
\end{enumerate}
\end{prop}

\begin{proof}
Let us show (1)-(3). 
By $\rho((X_{\overline k})^N) =\rho(X)=2$ and 
$\rho((Y_{\overline k})^N) =\rho(Y)=1$, 
it follows from \cite[Theorem 4.6]{Tan19} that $p=2$, 
$K_X^2 \in \{3, 4, 5, 6, 8\}$, and $K_Y^2 \in \{1, 2, 4\}$ (note that the case $(X_{\overline k})^N \simeq \P^1 \times \P^1$ does not occur, as 
$(X_{\overline k})^N$ has a nontrivial birational contraction).  
By $K_X^2 <K_Y^2$, we get  $(K_X^2, K_Y^2) =(3, 4)$. 
Therefore, 
$f$ is a blowup at  a $k$-rational point $Q$. 
Thus (1)-(3) holds. 

{Let us show (4). 
It is enough to find a prime divisor $C$ on $X$ satisfying $-K_X \sim C$ 
(\autoref{p hypersurf}). 
{Since $Q$ is a smooth $k$-rational point of $Y$ 
and $|-K_Y|$ is very ample 
{ (\autoref{p va deg3}, \autoref{l index 1}(b))}, 
there exists an effective Cartier divisor $C_Y$ on $Y$ 
such that $-K_Y \sim C_Y$ and $C_Y$ is smooth at $Q$. 
We see that $C_Y$ is a prime divisor by \autoref{l index 1}.}
For $E :=\Ex(f)$ and the proper transform $C := f_*^{-1}C_Y$ of $C_Y$ on $X$, 
we get $C+E = f^*C_Y \sim f^*(-K_Y) \sim -K_X+E,$ 
which implies $-K_X \sim C$. 
Thus (4) holds. 
}
\end{proof}

\begin{prop}[imprimitive, $Y$ geometrically normal]\label{p imprim geom-nor}
Assume that $k$ is separably closed. 
Let $f : X \to Y$ be a birational morphism of 
geometrically integral regular del Pezzo surfaces 
such that $\rho(X) = \rho(Y)+1$. 
Assume that $H^1(X, \MO_X)=0$, 
$X$ is not geometrically normal, 
and $Y$ is geometrically normal. 
Then $p=2$, $\rho(X)=2$, and $K_X^2 \in \{ 5, 6\}$. 
Moreover, the following hold. 
\begin{enumerate}
\item If $K_X^2 =5$, then $Y \simeq \P^2_k$ and 
$f : X \to Y$ is a blowup at a point of degree $4$. 
\item If $K_X^2 = 6$, then 
{$Y$ has a unique non-smooth point $P$, 
$P$ is of degree $2$,} 
$f: X \to Y$ is {the} blowup at {$P$}, 
and $Y$ is a non-smooth regular geometrically normal del Pezzo surface with $K_Y^2=8$. 
\end{enumerate} 
\end{prop}

\begin{proof}
By \cite[Theorem 4.6]{Tan19}, we have 
\[
\wt{X} := (X_{\overline k})^N = \P_{\P^1}(\MO \oplus \MO(n)) 
\]
for some   $n \in \{1,2, 4\}$. 
Since $Y$ is geometrically normal and $X$ is not, 
the divisorial part of the geometrically non-normal locus of $X$ must set-theoretically coincide with $E := \Ex(f)$. Again by \cite[Theorem 4.6]{Tan19}, one of the following holds. 
\begin{enumerate}
\item[(a)] $n=1$ and  $K_X^2=5$. 
\item[(b)] $n=2$ and $K_X^2 = 6$. 
\item[(c)] $n=4$ and $K_X^2 =8$. 
\end{enumerate}

Assume (c). 
Then $K_Y^2 =9$, i.e., $Y \simeq \P^2_k$. 
Moreover, $f: X \to Y$ must be a blowup at a $k$-rational point. 
Then $X$ is smooth, which is absurd.

Assume (a). 
We have birational morphisms 
\[
\mathbb F_1=\wt{X}=(X_{\overline k})^N \to X_{\overline k} \to 
Y_{\overline k}. 
\]
Since $Y_{\overline k}$ is normal, we get $Y_{\overline k}  \simeq \P^2_{\overline k}$. 
Since $Y$ is a Severi--Brauer surface over a separably closed field $k$, 
we get $Y = \P^2_k$.
The blowup centre $P$ for $f$ is a closed point of degree $4$. 

Assume (b). 
We have birational morphisms 
\[
\mathbb F_2=\wt{X}=(X_{\overline k} )^N \to X_{\overline k} \to 
Y_{\overline k}. 
\]
Then $Y_{\overline k}$ is the singular quadric surface in $\P^3_{\overline k}$, and hence 
$Y$ is a non-smooth regular geometrically normal del Pezzo surface with $K_Y^2=8$. 
The blowup centre $P$ is of degree $2$, 
{and is the unique non-smooth point {of $Y$}}. 
\end{proof}





\subsection{Classification: tables and proof}

\begin{lem}\label{l 1-4}
Let $X$ be a geometrically integral regular del Pezzo surface 
which is not geometrically normal. 
Assume that  $H^1(X, \MO_X)=0$, $\rho(X)=1$,  $K_X^2=4$, and $X(k) \neq \emptyset$. 
Then the following hold. 
\begin{enumerate}
\item $\Pic X = \Z K_X$. 
\item For a $k$-rational point $P$ and the blowup $\sigma : Y \to X$ at $P$, 
$Y$ is a geometrically integral regular del Pezzo surface 
which is not geometrically normal. 
\item $(X_{\overline k})^N \simeq \P^2_{\overline k}$. 
\end{enumerate}
\end{lem}

\begin{proof}
{The assertion (1) follows from the same argument as in \autoref{l index 1}(a).} 
The assertion (2) follows from (1) and \autoref{l rho=1}. 

Let us show (3). Since $\sigma : Y \to X$ is a birational morphism, 
so is the induced morphism 
\begin{equation}\label{e1 l 1-4}
(\sigma \times_k \overline k)^N : (Y  \times_k \overline k)^N \to (X  \times_k \overline k)^N. 
\end{equation}
By $K_X^2 =4$ and $K_Y^2 =3$, it follows from 
\cite[Theorem 4.6]{Tan19}  that 
\begin{itemize}
\item $p=2$, 
\item $(Y_{\overline k})^N \simeq \P_{\P^1}(\MO \oplus \MO(1))$ and 
\item $(X_{\overline k})^N \simeq \P^2$ 
or $(X_{\overline k})^N \simeq \P(1, 1, 4)$. 
\end{itemize}
Since (\ref{e1 l 1-4}) is a birational contraction, we get 
$(X_{\overline k})^N \simeq \P^2$. 
Thus (3) holds. 
\end{proof}

We conclude the classification of geometrically integral regular del Pezzo surfaces which are  geometrically non-normal. 

\begin{thm}\label{thm: class_dP}
Let $X$ be a geometrically integral regular del Pezzo surface. 
Assume that $X$ is not geometrically normal. 
Then $p \in \{2, 3\}$ and the following hold. 
\begin{enumerate}
\item If $p=3$, then 
{$X$ belongs to} \autoref{table p=3}. 
\begin{table}
\caption{$p=3$}\label{table p=3}
     \centering
{\renewcommand{\arraystretch}{1.35}%
      \begin{tabular}{|c|c|c|c|c|c|}
      \hline
{\rm No.}  & $\rho(X)$ & $K_X^2$& $h^1(\MO_X)$  & {\rm Properties} & {\rm Existence}\\      \hline
$1$-$1$  & $1$ & $1$ & $0$ & $X \simeq H_6 \subset \P(1, 1, 2, 3)$ & {\rm \autoref{e 1-1}} \\              
  & &  &  & 
{\rm and} $(X_{\overline k})^N \simeq \P^2_{\overline k}$ & \\            \hline
$1$-$3$  & $1$ & $3$ & $0$ & $X \simeq H_3 \subset \P^3_k$ &
{\rm \autoref{e 1-3}}\\                 
  & &  &  & 
{\rm and} $(X_{\overline k})^N \simeq \P(1, 1, 3)$ & \\         \hline
      \end{tabular}}
    \end{table}
\item If $p=2$, then 
{$X$ belongs to}  \autoref{table p=2}. 
\begin{table}
\caption{$p=2$}\label{table p=2}
     \centering
{\renewcommand{\arraystretch}{1.35}%
      \begin{tabular}{|c|c|c|c|c|c|}
      \hline
{\rm No.}  & $\rho(X)$ & $K_X^2$ & $h^1(\MO_X)$ & {\rm Properties} & {\rm Existence} \\      \hline
$1$-$1$  & $1$ & $1$ & $0$& $X \simeq H_6 \subset \P(1, 1, 2, 3)$ {\rm and} & {\rm \autoref{e 1-1}}\\        
  & &  &  & 
 $(X_{\overline k})^N \simeq \P^2$ & \\          \hline
$1$-$2$  & $1$ & $2$& $0$ & $X \simeq H_4 \subset \P(1, 1, 1, 2)$  {\rm and either} & {\rm \autoref{e 1-2 112}}\\         
  & &  &  & 
 $(X_{\overline k})^N \simeq \P(1, 1, 2)$ {\rm or} 
 $(X_{\overline k})^N \simeq \P^1 \times \P^1$  & {\rm \autoref{e 1-2}}\\                  \hline
$1$-$4$  & $1$ & $4$ & $0$& $X \simeq H_{2, 2} \subset \P^4_k$  {\rm and}& {\rm \autoref{e 1-4}}\\           
  & &  &  & 
 $(X_{\overline k})^N\simeq \P^2$ 
 & \\                \hline
$1$-$1$-{\rm i}  & $1$ & $1$ & $1$& $(X_{\overline k})^N \simeq \P^2$  & {\rm \autoref{e 1-1-i}} \\                     \hline
$1$-$2$-{\rm i}  & $1$ & $2$ & $1$& 
$(X_{\overline k})^N \simeq \P(1, 1, 2)$ 
& {\rm Unknown} \\          \hline 
$2$-$2$  & $2$ & $2$ & $0$& $X \simeq H_4 \subset \P(1, 1, 1, 2)$ {\rm and}  & {\rm \autoref{e 2-2}}\\        
 & &  &  & $(X_{\overline k})^N \simeq 
 \P^1 \times \P^1$ & \\       
& &  &  & 
 {\rm Types of extremal rays: $C+C$} & \\                     \hline 
$2$-$3$  & $2$ & $3$ & $0$& $X \simeq H_3 \subset \P^3_k$,  & {\rm \autoref{e 2-3}}\\        
  & &  &  & 
 {\rm $X$ is a blowup of $Y_{{\rm 1-4}}$ at a $k$-rational point,} & \\       
  & &  &  & 
{\rm and} $(X_{\overline k})^N \simeq \P_{\P^1}(\MO \oplus \MO(1))$ & \\       
& &  &  & 
 {\rm Types of extremal rays: $B+C$} & \\                     \hline 
$2$-$4$ & $2$ & $4$& $0$ & $X \simeq H_{2, 2} \subset \P^4_k$, 
 & {\rm \autoref{e 2-4}}\\   & &  &  & 
 {\rm $X$ is a double cover of $\P^1_{\SB, 1} \times \P^1_{\SB, 2}$}, & \\       
  & &  &  & 
{\rm and}  $(X_{\overline k})^N \simeq \P^1 \times \P^1$ & \\
& &  &  & 
 {\rm Types of extremal rays: $C+C$} & \\                               \hline   
$2$-$5$  & $2$ & $5$ & $0$& {\rm $X$ is a blowup of {$\P^2_k$} } & {\rm \autoref{e 2-5}}\\               
  & &  &  & 
{\rm at a purely inseparable point of degree $4$} & \\                          
  & &  &  & 
{\rm and} $(X_{\overline k})^N \simeq \P_{\P^1}(\MO \oplus \MO(1))$ & \\              
& &  &  & 
 {\rm Types of extremal rays: $B+C$} & \\                          \hline
$2$-$6$  & $2$ & $6$& $0$ & {\rm $X$ is a blowup of $Y_{{\rm 1-8}}$} & {\rm \autoref{e 2-6}}\\                  
  & &  &  & 
{\rm at a purely inseparable point of degree $2$} & \\             
  & &  &  & 
{\rm and} $(X_{\overline k})^N \simeq \P_{\P^1}(\MO \oplus \MO(2))$ & \\              
& &  &  & 
 {\rm Types of extremal rays: $B+C$} & \\                         \hline
      \end{tabular}}
    \end{table}
\end{enumerate}
In \autoref{table p=3} and \autoref{table p=2}, we use the notation listed in 
{\rm (i)}-{\rm (vi)}. 
\begin{enumerate}
\renewcommand{\labelenumi}{(\roman{enumi})}
\item 
The column {\rm No.} gives the numbering. 
When $H^1(X, \MO_X)=0$ (resp. $H^1(X, \MO_X) \neq 0$), 
the numbering  ${\rm a}$-${\rm b}$  (resp. ${\rm a}$-${\rm b}$-${\rm i}$) is given by $a :=\rho(X)$ and $b:=K_X^2$ 
({\lq\,${\rm i}$\rq}\, stands for irregular). 
\item 
$H_d$ denotes a weighted hypersurface of 
degree $d$. 
$H_{2, 2} \subset \P^4_k$ is a complete intersection of two quadric hypersurfaces. 
\item 
$Y_{{\rm 1-4}}$ is a geometrically integral geometrically non-normal regular del Pezzo surface with $\rho(Y_{{\rm 1-4}}) =1$ and $K^2_{Y_{{\rm 1-4}}}=4$. 
\item 
$Y_{{\rm 1-8}}$ is a {non-smooth} geometrically normal regular del Pezzo surface with $\rho(Y_{{\rm 1-8}}) =1$ and $K^2_{Y_{{\rm 1-8}}}=8$. 
\item 
$\P^1_{\SB, 1}$ and $\P^1_{\SB, 2}$ are 
one-dimensional Severi-Brauer varieties over $k$, i.e., 
each $\P^1_{\SB, i}$ is a smooth conic in $\P^2_k$. 
\item When $\rho(X)=2$, an extremal ray $R$ is said to be of type $B$ (resp. $C$) if 
$\dim Z = 2$ (resp. $\dim Z =1$) for the contraction $X \to Z$ of $R$ 
({\lq\,$B$\rq}\, stands for birational and {\lq\,$C$\rq}\, stands for conic bundles). 
For example, if $\rho(X) =2$ and {$K_X^2 \not\in \{2, 4\}$}, 
then there is an extremal ray of type $B$ and the other extremal ray is of type $C$. 
\end{enumerate}
\end{thm}

\begin{rem}
\begin{enumerate}
\item[(a)] 
{If $k$ is a $C_1$-field or separably closed, 
then we get $\P^1_{\SB, i} \simeq \P^1_k$, as $\P^1_{\SB, i}$ is 
a smooth conic on $\P^2_k$.}
\item[(b)] 
If $k$ is separably closed, then 
$\rho(X) = \rho((X_{\overline k})^N)$, 
and hence we automatically have 
$(X_{\overline k})^N \simeq \P(1, 1, 2)$ 
for No. 1-2. 
\end{enumerate}   
\end{rem}


\begin{proof}[Proof of \autoref{thm: class_dP}]
If $H^1(X, \MO_X) \neq 0$, then 
{$X$ belongs to one of the cases 1-1-i and 1-2-i}  
by \cite[Theorem 4.6]{Tan19} and \cite[Theorem 1.1]{BM23}. 
In what follows, we assume $H^1(X, \MO_X)=0$. 
The types of the extremal rays are determined by \cite[Proposition 4.32 and Proposition 4.35]{BFSZ}.

\setcounter{step}{0}

\begin{step}\label{s1 class_dP}
The assertion of Theorem \ref{thm: class_dP} holds if $k$ is separably closed. 
\end{step}

\begin{s1proof}
{\em Primitive case.} 
Assume that $X$ is primitive. 
Then 
either 
\begin{enumerate}
    \item[(I)] $\rho(X)=1$, or 
    \item[(II)] $\rho(X)=2$ and both extremal rays induce  conic bundle structures.  
\end{enumerate}

Assume $\rho(X)=1$. 
Then 
$K_X^2$ and $(X_{\overline k})^N$ satisfy {the conditions of one of the cases  1-1, 1-2, 1-3, 1-4}
by \autoref{p rho 1} {and \autoref{l 1-4}}. 
Since there is a prime divisor $C$ on $X$ with $-K_X \sim C$ (\autoref{l index 1}), 
we get the descriptions in the column \lq\lq Properties" by 
\autoref{p hypersurf}. 

Assume $\rho(X)=2$. 
Since $X$ is primitive, both extremal rays induce conic bundle structures. 
{By \autoref{p 2 conic bdls}, we get $p=2$ {and $K_X^2 \in \{2, 4\}$.} 
If $K_X^2 =2$ (No. 2-2), then $X \simeq H_4 \subset \P(1, 1, 1, 2)$ 
holds by \autoref{p hypersurf} and \autoref{p 2 conic bdls}(1). 
If $K_X^2 =4$ (No. 2-4), then we have the double cover $\pi : X \to \P^1_k\times \P^1_k$ (\autoref{p 2 conic bdls}(2)).} 
Since $|-K_X|$ is  base point free (\autoref{p 2 conic bdls}), 
we get $X \simeq H_{2, 2} \subset \P^4_k$ (\autoref{p hypersurf}). 

\medskip
{\em Imprimitive case.} 
Assume that $X$ is imprimitive. 
Let $\sigma : X \to Y$ be a birational morphism to a geometrically integral  regular del Pezzo surface $Y$. 
If $Y$ is geometrically non-normal (resp. geometrically normal), 
then the assertion follows from \autoref{p imprim geom-non} (resp. \autoref{p imprim geom-nor}). 
This completes the proof of Step \ref{s1 class_dP}. 
\end{s1proof}

\begin{step}\label{s2 class_dP}
The assertion of \autoref{thm: class_dP} holds 
when $p=2$.
\end{step}

\begin{s2proof}
Assume $p=2$. 
Set $X^{\sep} := X \times_k k^{\sep}$. 
Then the assertion holds for $X^{\sep}$ by Step \ref{s1 class_dP}. 
In particular, the description for  
$(X_{\overline k})^N (\simeq (X^{\sep} \times_{k^{\sep}} \overline k)^N$) holds.

Let us prove the following. 
\begin{enumerate}
\item[$(\alpha)$] 
It holds that $\rho(X) \leq \rho(X^{\sep})$ and $K_X^2 =K^2_{X^{\sep}}$. 
Moreover, if  $K_X^2 =1$, 
then $\rho(X) = \rho(X^{\sep}) =1$. 
\item[($\beta$)] If $\rho(X)=2$, then the types of the extremal rays of $X$ and $X^{\sep}$ coincide. 
\item[$(\gamma)$] 
If $1 \leq K_X^2 \leq 4$, then 
there exists a prime divisor $C$ satisfying $-K_X \sim C$, and hence \autoref{p hypersurf} is applicable. 
\end{enumerate}

Let us show ($\alpha$). 
By standard argument, we have 
$\rho(X) \leq \rho(X^{\sep})$ and $K_X^2 =K^2_{X^{\sep}}$. 
When $K_X^2=1$, we get $\rho(X) \leq \rho(X^{\sep}) = 1$ (Step \ref{s1 class_dP}), which implies $\rho(X) = \rho(X^{\sep}) = 1$. 
Thus ($\alpha$) holds.  
Let us show ($\beta$). 
Assume $\rho(X)=2$. 
We get $\rho(X^{\sep})=2$ by ({\cyan $\alpha$}) and \cite[Theorem 4.6]{Tan19}. 
Let $\sigma : X \to Y$ be the contraction of an extremal ray. 
By $\rho(X^{\sep}) =2$, its base change $\sigma \times_k k^{\sep}$ is the contraction of an extremal ray. 
Since each of $X$ and $X^{\sep}$ has exactly two such contractions, 
the contractions of $X$ and $X^{\sep}$ 
are corresponding via the base change $(-) \times_k k^{\sep}$. 
Thus $(\beta)$ holds. 

Let us show $(\gamma)$. 
If $K_X^2 \in \{3, 4\}$, then $|-K_X|$ is very ample, because so is its base change $|-K_{X^{\sep}}|$ (Step \ref{s1 class_dP}); we are done by a Bertini theorem \cite[Theorem 7' in page 376, cf. Theorem 7 in page 368]{Sei50}. 
Hence we may assume $K_X^2 \in \{1, 2\}$. 
Let us treat the case when $\rho(X^{\sep})=1$. 
Take an effective divisor $D$ with $D \sim -K_X$ (\autoref{lem: Kod_vanishing}). 
Then $D \times_k k^{\sep}$ is a prime divisor by  \autoref{l index 1} (b), and hence so is $D$. 
The problem is reduced to the case when $K_X^2 \in \{1, 2\}$ and $\rho(X^{\sep})=2$. 
We have $K_X^2 =2$ (Step \ref{s1 class_dP}). 
Then, for an effective divisor $D$ with $D \sim -K_X$, 
$D \times_k k^{\sep}$ is a prime divisor (\autoref{p 2 conic bdls}), and hence so is $D$. 
This completes the proof of $(\gamma)$. 

\medskip

Assume $K_X^2  =K^2_{X^{\sep}} =1$ (resp. $K_X^2  =K^2_{X^{\sep}} =2$). 
Then the assertion follows from $(\gamma)$ (resp. $(\alpha)$-$(\gamma)$). 

Assume $K_X^2  =K^2_{X^{\sep}} \in \{ 3, 5, 6\}$.
Since $\rho(X^{\sep})=2$ and the types of the extremal rays of $\NE(X^{\sep})$ are $B+C$ (Step \ref{s1 class_dP}), 
$X^{\sep}$ has a unique curve $C'$ with $C'^2<0$. 
By Galois descent, 
{the image $C$ of $C'$ on $X$ is a curve satisfying $C^2 <0$.} 
In particular, $\rho(X)=2$. 
By $(\beta)$, there is a birational morphism 
$\sigma : X \to Y$ to a projective birational morphism to a 
regular del Pezzo surface $Y$ which contracts $C$. 
Then this is a blowup of some closed subscheme $Z$ of $Y$ \cite[Theorem II.7.17]{Ha77}. 
{Since $\sigma: X \to Y$ is a birational morphism of regular projective surfaces satisfying $\rho(X) = \rho(Y)+1$, 
we may assume that $Z$ is a closed point on $Y$.} 
Since blowups commute with a flat base change $(-) \times_k k^{\sep}$, 
$Z \times_k k^{\sep}$ is a purely inseparable point, 
and hence so is $Z$.  
Moreover, it follows from \cite[Proposition 4.32]{BFSZ} that $Y \simeq \P^2_k$ (resp. $K_Y^2=8$, $K_Y^2=4$) when $K_X^2 =5$ (resp. $K_X^2=6, K_X^2=3$). 
This completes the proof for the case when $K_X^2 \in \{3, 5, 6\}$.

Assume $K_X^2  =K^2_{X^{\sep}} =4$. 
By ($\gamma$), we get $X \simeq H_{2, 2} \subset \P^4_k$. 
{We now treat the case when $\rho(X)=1$. 
In this case, it suffices to exclude the possibility $(X_{\overline k})^N \simeq  \P^1_{\overline k} \times \P^1_{\overline k}$. 
If $(X_{\overline k})^N \simeq \P^1_{\overline k} \times \P^1_{\overline k}$, then 
we get $\rho((X_{\overline k})^N) =2$ and the Galois group 
$G := {\rm Gal}(\overline{k}/k^{1/p^{\infty}})$  permutes two Mori fibre space structures on $(X_{\overline k})^N$, 
which contradicts the fact that 
the isomorphism class of the invertible sheaf
 $\omega_{(X_{\overline k})^N} \otimes f^*\omega_X^{-1} \simeq \MO_{\P^1 \times \P^1}(1, 0)$ \cite[Theorem 4.6]{Tan19} is not stable under 
the involution switching the direct product factors, 
where $f : (X_{\overline k})^N \to X$ denotes the induced morphism.} 
In what follows, we assume  $\rho(X)=2$. 
By ($\beta$) and Step \autoref{s1 class_dP},  
there are two contractions $\pi_1 :X \to B_1$ and $\pi_2 : X \to B_2$ such that $\dim B_1 = \dim B_2=1$. 
Set $\pi : = \pi_1 \times \pi_2 : X \to B_1 \times_k B_2$. 
Taking the base change $(-) \times_k k^{\sep}$, 
each $\pi_i \times_k k^{\sep} : X^{\sep} \to B_i \times_k k^{\sep}$ 
is  a contraction of an extremal ray. 
Hence $B_i \times_k k^{\sep} \simeq \P^1_{k^{\sep}}$. 
By \autoref{p 2 conic bdls}, $\pi : X \to B_1 \times B_2$ is a double cover. 
This completes the proof of Step \ref{s2 class_dP}. 
\end{s2proof}

\begin{step}\label{s3 class_dP}
The assertion of Theorem \ref{thm: class_dP} holds 
when $p=3$. 
\end{step}

\begin{s3proof}
Assume $p=3$. 
Set $X^{\sep} := X \times_k k^{\sep}$. 
Then the assertion holds for $X^{\sep}$ by Step \ref{s1 class_dP}. 
In particular, the description for  
$(X_{\overline k})^N (\simeq (X^{\sep} \times_{k^{\sep}} \overline k)^N$) holds. 
In order to prove the assertion of Step \ref{s3 class_dP}, 
it is enough to show $(\alpha)'$ and $(\gamma)'$ below. 
\begin{enumerate}
\item[$(\alpha)'$]
$\rho(X) = \rho(X^{\sep})=1$ and $K_X^2 =K^2_{X^{\sep}}$. 
\item[$(\gamma)'$] 
There exists a prime divisor $C$ satisfying $-K_X \sim C$, and hence \autoref{p hypersurf} is applicable. 
\end{enumerate}
The assertion $(\alpha)'$ follows from $\rho(X) \leq \rho(X^{\sep})=1$. 
Let us show $(\gamma)'$. 
As $K_X^2 =K^2_{X^{\sep}} \in \left\{1, 3 \right\}$ and 
$\Pic X$ is a free $\Z$-module of rank $1$, 
we deduce that $\Pic X=\mathbb{Z}K_X$. 
The Riemann--Roch theorem implies that $h^0(X,\MO_X(-K_X))>0$, 
and hence there exists an effective divisor $C$ such that $-K_X \sim C$. 
By $\Pic X = \Z K_X$, 
$C$ is a prime divisor and thus $(\gamma)'$ holds. 
This completes the proof of Step \ref{s3 class_dP}.
\end{s3proof}

\medskip

Step \ref{s2 class_dP} and 
Step \ref{s3 class_dP} complete the proof of \autoref{thm: class_dP}. 
\end{proof}

\begin{cor} \label{cor: non-normal-rho2}
Assume that $k$ is a $C_1$-field. 
Let $X$ be a geometrically integral regular del Pezzo surface 
which is not geometrically normal. 
Then  the following hold. 
\begin{enumerate}
\item $p=2$. 
\item $\rho(X) = 2$ and $H^1(X, \MO_X)=0$. 
\item $K_X^2 \in \{2, 4,  5, 6\}$. 
\item $X(k) \neq \emptyset$. 
\end{enumerate}
\end{cor}

\begin{proof}
As a $C_1$-field has $p$-degree at most 1 \cite[Lemma 6.1]{BT22},  \cite[Theorem 14.1]{FS20} implies that $\rho(X) \geq 2$ (and hence $\rho(X)=2$ by \autoref{thm: class_dP}) 
and $H^1(X, \MO_X)=0$. 
The case  of No. 2-3 in \autoref{thm: class_dP} does not occur, as otherwise 
the case of No. 1-4 would occur, which is absurd. 
Thus $X$ is of No. 2-2, 2-4, 2-5, or 2-6. 
In particular, (1)-(3) hold. 
The assertion (4) easily follows from \autoref{thm: class_dP}. 
\qedhere



\end{proof}

\subsection{Conic bundles}

In \autoref{thm: class_dP}, 
each class with $\rho(X)=2$ (i.e., No. 2-x) has a conic bundle structure. 
We here investigate their properties.


\begin{prop}\label{p always wild}
Let $X$ be a geometrically integral regular del Pezzo surface with $\rho(X) \geq 2$. 
Then the following are equivalent. 
\begin{enumerate}
\item $X$ is not geometrically normal. In particular, $\rho(X)=2$ (\autoref{thm: class_dP}). 
\item 
{There exists a morphism  $\pi : X \to B$ 
such that $B$ is a regular projective curve, $\pi_*\MO_X = \MO_B$, 
and 
no fibre of $\pi$ is smooth.} 
\end{enumerate}
\end{prop}

\begin{proof}
Let us show the implication (1) $\Leftarrow$ (2). 
Assume that (1) does not hold. 
Then $X$ is geometrically normal. 
Let $\pi : X \to B$ be a morphism 
such that $B$ is a regular projective curve and $\pi_*\MO_X = \MO_B$. 
It is enough to show that $\pi$ is generically smooth.
For the base change $\pi_{\overline k}:= \pi \times_k \overline k : X_{\overline k} \to B_{\overline k}$, 
the induced morphism 
$\pi_{\overline k}|_{(\pi_{\overline k})^{-1}(U)} : (\pi_{\overline k})^{-1}(U) \to U$ is a 
Mori fibre space from a smooth surface 
$(\pi_{\overline k})^{-1}(U)$ 
to a smooth curve $U$ 
for some non-empty open subset $U$ of $B_{\overline k}$. 
Then the generic fibre of $\pi_{\overline k}|_{(\pi_{\overline k})^{-1}(U)}$ is a  geometrically integral conic (cf. \cite[Theorem 7.1]{Bad01}, \cite[Lemma 10.6]{kk-singbook}). 
Therefore, 
$\pi_{\overline k}$ is generically smooth, i.e., 
(2) does not hold. 
This complete the proof of the implication (1) $\Leftarrow$ (2). 

Let us show the implication (1) $\Rightarrow$ (2). 
Assume (1). 
By \autoref{thm: class_dP}, 
we get $\rho(X)=2$ and there exists a Mori fibre space $\pi : X \to B$ with $\dim B =1$. 
We treat the case when 
{$K_X^2 \not\in \{2, 4\}$}, i.e., $K_X^2 \in \{3, 5, 6\}$ 
(\autoref{thm: class_dP}). 
In order to prove that $\pi$ has no smooth fibre, 
we may assume that $k$ is separably closed. 
In particular, $B = \P^1_k$. 
We have the following commutative diagram: 
\[
\begin{tikzcd}
(X_{\overline k})^N \arrow[r, "\nu"] \arrow[d, "\pi''"] 
\arrow[rr, bend left, "g"]&
X_{\overline k} \arrow[r, "\alpha"] \arrow[d, "\pi'"] &
X\arrow[d, "\pi"]\\
B'' \arrow[r] & B' :=B \times_k \overline k =\P^1_{\overline k} \arrow[r] &B =\P^1_k, 
\end{tikzcd}
\]
where the right square is cartesian, 
$\nu : (X_{\overline k})^N \to X_{\overline k}$ is the normalisation, 
and $B''$ is the Stein factorisation 
of the composition $(X_{\overline k})^N \xrightarrow{\nu} X_{\overline k} \xrightarrow{\pi'} B' = B \times_k \overline k$. 
By \cite[Theorem 4.6]{Tan19}, the conductor divisor $D$ of $\nu$,  
given by 
\[
K_{(X_{\overline k})^N} +D =g^*K_X,
\]
satisfies $D \cdot F'' >0$ for a fibre $F''$ of $\pi''$, 
i.e., $D$ dominates $B''$. 
Therefore, the singular locus of $X_{\overline k}$ 
dominates $B' = B \times_k \overline k$, as it contains the image of $D$. 
Hence the non-smooth locus $\Sigma$ of $X$ dominates $B$. 
In particular, $\pi$ is not generically smooth. 

We may assume that {$K_X^2 \in \{2, 4\}$.} 
In this case, we have two Mori fibre spaces $\pi_1 : X \to B_1$ and $\pi_2 : X \to B_2$ (\autoref{thm: class_dP}). 
By the same argument as above, we see that one of the contraction is not generically smooth (indeed, the conductor divisor dominates $B_1$ or $B_2$ \cite[Theorem 4.6]{Tan19}). 
Thus (2) holds. 
\end{proof}

\begin{rem}
Let $X$ be a geometrically integral regular del Pezzo surface. 
Assume that $\rho(X) = 2$, $K_X^2 =4$, and  $X$ is not geometrically normal. 
We have 
the contractions 
$\pi_1 : X \to B_1$ and $\pi_2 : X \to B_2$ 
of the extremal rays. 
By the proof of \autoref{p always wild}, one of $\pi_1$ and $\pi_2$ is generically smooth and the other is not. 
\end{rem}

\begin{rem}
Assume that $k$ is of characteristic two. 
Let $\pi : X \to B$ be a Mori fibre space, 
where $X$ is a regular projective surface and $B$ is a regular projective curve. 
Assume that $\pi$ is not generically smooth. 
Then  
$\pi^{-1}(b)$ is  integral but not geometrically reduced  for every point $b \in B$ \cite[Proposition 2.18]{BT22} (e.g., $\pi^{-1}(b)=\{ x^2 +ty^2 =0\} \subset \P^2_{\F_2(t)} = \Proj \F_2(t)[x, y, z]$). 
An explicit example is given in \autoref{e 2-5}. 
\end{rem}

\subsection{Examples}

We collect examples of geometrically non-normal {geometrically integral} regular del Pezzo surfaces, according to degree and characteristic. 
This shows that all the cases  in \autoref{thm: class_dP} 
are realised except possibly for No. 1-2-i. 
Throughout this subsection, we use the following notation. 

\begin{notation}
\begin{enumerate}
\item $\partial_x$ denotes the derivation $\partial/\partial x$, e.g., 
$\partial_x (x^3+xy) = 3x^2 +y$. 
\item In order to specify the weighted homogeneous coordinate, 
we set  
\[
\P(d_0, ..., d_n)_{[x_0:\cdots : x_n]} := \Proj k[x_0, ..., x_n], 
\]
where $k[x_0, ..., x_n]$ is the polynomial ring over $k$ such that 
$\deg x_i =d_i$ for every $i$. 
Similarly, $\P^n_{[x_0:\cdots : x_n]} := \mathbb{P}(1, ..., 1)_{[x_0:\cdots : x_n]}= \Proj k[x_0, ..., x_n]$,  $\A^n_{(x_1, ..., x_n)} := \Spec k[x_1, ..., x_n]$, and $\A^1_x := \Spec k[x]$. 
\end{enumerate}
\end{notation}

{
\begin{prop}\label{p Jacobian}
Let $\F$ be an algebraically closed field 
and set $k := \F(s_1, ..., s_r)$, which is a  purely transcendental extension of $\F$ with variables $s_1, ..., s_r$. 
For a homogeneous polynomial $f \in k[x_0, ..., x_N]$, we set 
\[
X := \Proj\, k[x_0, ..., x_N]/(f). 
\]
Take a closed point $P$ of $\P^N_k$. 
Then the following are equivalent. 
\begin{enumerate}
    \item[(1)] $P \in X$ and $X$ is not regular at $P$. 
    \item[(2)] $P \in \{ f = \partial_{s_1}f = \cdots = \partial_{s_r}f = 
\partial_{x_0}f = \cdots = \partial_{x_N} f =0\}$. 
\end{enumerate}
In other words,  $\{ f = \partial_{s_1}f = \cdots = \partial_{s_r}f = 
\partial_{x_0}f = \cdots = \partial_{x_N} f =0\}$ is the non-regular locus of $X$. 
\end{prop}

This result is well known when the base field $k$ is algebraically closed. 
We give a proof for the sake of completeness. 

\begin{proof}
We may assume $P \in D_+(x_0)$ by symmetry. 
For $g:= f(1, x_1, ..., x_n)$, we get (1) $\Leftrightarrow$ (1)'   by the Jacobian criterion for hypersurfaces:
\begin{enumerate}
\item[(1)'] $P \in \{ g = \partial_{s_1}g = \cdots = \partial_{s_r}g = 
\partial_{x_1}g = \cdots = \partial_{x_N}g =0\}$. 
\end{enumerate}
Since each of $\partial_{s_j}$ and $\partial_{x_i}$ (with $i>0$) commutes with the substitution $x_0=1$, we get  
the equivalence (1)'  $\Leftrightarrow$ (2)'. 
\begin{enumerate}
\item[(2)'] $P \in \{ f = \partial_{s_1}f = \cdots = \partial_{s_r}f = \partial_{x_1}f = \cdots =\partial_{x_N} f =0\}$. 
\end{enumerate}
As the implication (2) $\Rightarrow$ (2)' is obvious, 
it is enough to show the opposite one (2) $\Leftarrow$ (2)', 
which follows from the Euler identity $\partial_{x_0}f + \partial_{x_1}f + \cdots + \partial_{x_n}f = (\deg f)f$.     
\end{proof}
\begin{rem}\label{r Jacobian}
Let $\F$ be an algebraically closed field 
and set $k := \F(s_1, ..., s_r)$, which is a  purely transcendental extension of $\F$ with variables $s_1, ..., s_r$. 
Let $k[x_0, ..., x_N]$ be the weighted polynomial ring 
that satisfies $d_i := \deg x_i \in \Z_{>0}$ and $\deg c =0$ for every $c \in k^{\times}$. 
For a weighted homogeneous polynomial $f \in k[x_0, ..., x_N]$, we set 
\[
X := \Proj\, k[x_0, ..., x_N]/(f) \subset \P(d_0, ..., d_N)_{[x_0: \cdots : x_N]}. 
\]
Assume that there exists an integer $r \geq 0$ such that 
$d_0 = \cdots d_r =1$ and $X \subset D_+(x_0) \cup \cdots \cup D_+(x_r)$. 
By the same proof as in \autoref{p Jacobian}, 
the non-regular locus $X_{\text{non-reg}}$ of $X$ is given by 
\[
X_{\text{non-reg}} = 
\{ f = \partial_{s_1}f = \cdots = \partial_{s_r}f = 
\partial_{x_0}f = \cdots = \partial_{x_N} f =0\}. 
\]
\end{rem}
}

\begin{ex}[$p \in \{2, 3\}, \rho(X)=1, K_X^2=1$]\label{e 1-1}
Let $\F$ be an algebraically closed field of characteristic three. 
Set $k :=\mathbb F(s_0, s_1, s_2, s_3)$, i.e., 
$k$ is a purely transcendental extension of $\F$ with four variables 
$s_0, s_1, s_2, s_3$. 
We define 
\[
X:= \{ s_0z^2 + s_1y^3 + s_2x_0^6 + s_3x_1^6 =0\} := 
\Proj \frac{k[x_0, x_1, y, z]}
{(s_0z^2 + s_1y^3 + s_2x_0^6 + s_3x_1^6)}
\]
\[
\subset 
\Proj k[x_0, x_1, y, z] = \P(1, 1, 2, 3)_{[x_0: x_1: y: z]}. 
\]
Note that 
we have $[0:0:0:1], [0:0:1:0] \not\in X$, 
where $[0:0:0:1], [0:0:1:0]$ are all the singular points of 
$\P(1, 1, 2, 3)$. 
In what follows, we prove the following. 
\begin{enumerate}
\item $X$ is regular. 
\item $X$ is geometrically integral. 
\item $X$ is not geometrically normal.  
\item  $\rho(X) =1$ and $K_X^2 =1$. 
\end{enumerate}

Let us show (1). 
Substituting $x_0=1$, we get 
\begin{equation}\label{e1 1-1}
D_+(x_0) = \{ s_0z^2 + s_1y^3 + s_2 + s_3x_1^6 =0\}  \subset 
\A^3_{(x_1, y, z)} = \Spec k[x_1, y, z]. 
\end{equation}
By the Jacobian criterion for smoothness, 
the affine hypersurface 
\begin{equation}\label{e2 1-1}
\{ s_0z^2 + s_1y^3 + s_2 + s_3x_1^6 =0\} \subset \A^7_{\F} = \Spec \F[s_0, s_1, s_2, t, x_1, y, z]
\end{equation}
in $\A^7_{\F}$ is smooth over $\F$ (consider $\partial_{s_2} := \partial/\partial s_2$). 
Since $D_+(x_0)$ in (\ref{e1 1-1}) is obtained by 
applying the localisation $(-) \otimes_{\F[s_0, s_1, s_2, s_3]} \F(s_0, s_1, s_2, s_3)$, 
$D_+(x_0)$ is regular. 
By symmetry, $D_+(x_1)$ is regular. 
On the remaining locus $U :=\{yz \neq 0\}$,
the following holds \cite[Section 3.1]{Oka21}: 
\[
U = \{ s_0u^2 + s_1u^3 + s_2x_0^6 + s_3x_1^6 =0\} \subset 
\mathbb A^2_{(x_0, x_1)} \times (\A^1_{u} \setminus \{0\}). 
\]
By applying Jacobian criterison for smoothness as above (specifically, use  $\partial_{s_0}$), 
we get $u=0$, which is absurd. 
Hence $X$ is regular, i.e., (1) holds. 

Let us show (3). 
We get 
\[
 X_{\overline k} \simeq  \{ z^2 + y^3 + x_0^6 + x_1^6 =0\} \subset 
\P(1, 1, 2, 3)_{[x_0: x_1: y: z]}. 
\]
Replacing $y+x_0^2+x_1^2$ by $y$, we get 
\[
X_{\overline{k}} \simeq  \{ z^2 + y^3=0\} \subset 
\P(1, 1, 2, 3)_{[x_0: x_1: y: z]}. 
\]
As $
D_+(x_0) = \{ z^2 + y^3=0\} \subset \A^3_{(x_1, y, z)},$
the singular locus of $X_{\overline k}$ contains the affine curve 
$\{ z^2 +y^3 = 2z=0\} \subset \A^3_{(x_1, y, z)}$. Hence $X$ is not geometrically normal. 
Thus (3) holds. 
Since $X$ is a regular projective variety with $H^0(X, \MO_X)=k$, 
$X_{\overline k}$ is irreducible. 
In order to show (2), 
it is enough to prove that $X_{\overline k}$ is reduced. 
Since $X_{\overline k}$ is Cohen-Macaulay, it suffices to find a smooth point of $X_{\overline k}$ (cf. reduced $\Leftrightarrow$ $R_1 +S_0$). 
This follows from the fact that 
 $D_+(x_0) \subset X_{\overline k}$ is smooth at 
$(x_0, y, z) = (0, -1, 1) \in D_+(x_0)$ by Jacobian criterion. 
Thus (2) holds. 
Since $X$ is a hypersurface on $\P(1, 1, 2, 3)$ of degree $6$, 
we get $K_X^2 =1$. 
Then \autoref{thm: class_dP} implies $\rho(X)=1$. 
Thus (4) holds.

\medskip

For an algebraically closed field $\F$ of characteristic two and 
$k :=\mathbb F(s_0, s_1, s_2, s_3)$, the same equation 
\[
X:= \{ s_0z^2 + s_1y^3 + s_2x_0^6 + s_3x_1^6 =0\} \subset 
\P(1, 1, 2, 3)_{[x_0: x_1: y: z]}. 
\]
gives an example satisfying (1)-(4). 
We omit the proof, as it is identical to the above. 
\end{ex}

\begin{ex}[$p=2, \rho(X)=1, K_X^2=2, (X_{\overline{k}})^N \simeq \mathbb{P}(1,1,2) $]\label{e 1-2 112}
Let $\F$ be an algebraically closed field of characteristic two. Set $k:= \mathbb{F}(s,t)$ and
\[
X \coloneqq \{ f\coloneqq y^2+x_0^3x_1+sx_1^4+tx_2^4=0 \} \subset \mathbb{P}(1,1,1,2)_{[x_0:x_1:x_2:y]}.
\]
In what follows, we prove the following. 
\begin{enumerate}
\item $X$ is regular. 
\item $X$ is geometrically integral.  
\item $X$ is not geometrically normal. 
\item $(X_{\overline{k}})^N \simeq \mathbb{P}(1,1,2)$ for the normalisation 
$(X_{\overline{k}})^N$ of $X_{\overline k}$. 
\item $\rho(X) =1$ and $K_X^2 =2$. 
\end{enumerate}

Let us show (1). 
By $[0:0:0:1] \not\in X$, 
we have $X \subset D_+(x_0) \cup D_+(x_1) \cup D_+(x_2)$. 
By \autoref{r Jacobian}, 
it is enough to prove 
\[
\{ f = \partial_{x_0}f = \partial_{x_1}f = \partial_{x_2}f = \partial_y f 
 = \partial_s f  = \partial_t f =0\} = \emptyset. 
\]
This follows from $\partial_sf = x_1^4, \partial_t f = x_2^4, \partial_{x_1}f = x_0^3$, and  $f = y^2+x_0^3x_1+sx_1^4+tx_2^4$. 
Thus (1) holds.


Let us show (2) and (3). 
Over $\overline{k}$, we get 
\[
f= y^2+x_0^3x_1+sx_1^4+tx_2^4
= (y + s^{1/2}x^2_1 + t^{1/2}x_2^2)^2 + x_0^3x_1, 
\]
and hence 
\[
X_{\overline{k}} \simeq \{ g := y^2+x_0^3x_1=0 \} \subset \mathbb{P}(1,1,1,2)_{[x_0:x_1:x_2:y]}.
\]
Since the right hand side is contained in $D_+(x_0) \cup D_+(x_1) \cup D_+(x_2)$, 
we may apply Remark \ref{r Jacobian} again. 
We have 
$(\partial_{x_0}g, \partial_{x_1}g,\partial_{x_2}g,\partial_{y}g)=(x_0^2x_1,x_0^3,0,0)$. 
In particular, $[1:0:0:0]$ is a smooth point of $X$, and hence $X_{\overline{k}}$ is integral. 
Thus (2) holds. 
At the same time, the non-smooth locus 
$(X_{\overline k})_{\text{non-reg}}$ of $X_{\overline k}$ is given by 
$(X_{\overline k})_{\text{non-reg}} = \{ x_0=y=0\}$, and hence one-dimensional. 
Thus (3) holds. 

Let us show (4). 
It is enough to show that the morphism 
\[
\nu \colon \mathbb{P}(1,1,2)_{[z_0:z_1:w]} \to X_{\overline{k}}, \qquad 
[z_0:z_1:w] \mapsto [z_0^2: z_1^2 : w : z_0^3z_1]
\]
is the normalisation of $X_{\overline k}$. 
For a closed point $[a_0 : a_1 : a_2: b] \in X_{\overline k}$, 
it is easy to see that $\nu^{-1}([a_0 : a_1 : a_2: b])$ is a finite set, and hence $\nu$ is a finite morphism. 
We have $\nu^{-1}(D_+(x_1)) = D_+(z_1)$, and the induced morphism 
between these charts is given as follows: 
\[
D_+(z_1) = \mathbb A^2_{\overline k} = \Spec \overline k [z_0, w] 
\to \Spec \overline k[x_0, x_2, y]/(y^2 + x_0^3), \qquad (z_0, w) \mapsto 
(z_0^2, w, z_0^3). 
\]
Hence $\nu$ is birational. 
Thus (4) holds. 

Let us show (5). By $\rho(X) \leq \rho( (X_{\overline k})^N) =1$, 
we get $\rho(X)=1$. 
It holds that $K_X^2 =2$ by adjunction formula. 
Thus (5) holds. 
\end{ex}

\setcounter{equation}{0}

\begin{ex}[$p=2, \rho(X)=1, K_X^2=2$, $(X_{\overline k})^N \simeq \P^1 \times \P^1$]\label{e 1-2}
Let $\F$ be an algebraically closed field of characteristic two. 
Set $k :=\mathbb F(s_0, s_1, s_2, t)$ and 
\[
X:= \{  y^2 + t x_0^2y + s_0x_0^4 + s_1x_1^4 + s_2x_2^4 =0\} \subset 
\P(1, 1, 1, 2)_{[x_0:x_1:x_2: y]}. 
\]
We have $[0:0:0:1] \not\in X$ for the unique singular point $[0:0:0:1]$ of $\P(1, 1, 1, 2)$. 
Hence $X \subset D_+(x_0) \cup D_+(x_1) \cup D_+(x_2)$. In what follows, we prove the following. 
\begin{enumerate}
\item $X$ is regular. 
\item $X$ is geometrically integral. 
\item $X$ is not geometrically normal. 
\item $(X_{\overline k})^N \simeq \P^1_{\overline k} \times \P^1_{\overline k}$. 
\item $\rho(X)=1$ and $K_X^2=2$.
\end{enumerate}


Let us show (1). 
By $X \subset D_+(x_0) \cup D_+(x_1) \cup D_+(x_2)$, 
we may apply \autoref{r Jacobian}. 
Then $X$ is regular by $\{ f = \partial_{s_0}f = \partial_{s_1}f =\partial_{s_2}f =0\} = \emptyset$. 
Thus (1) holds.


Let us show (2) and (3). 
After taking a suitable coordinate change, we have 
\[
X_{\overline k} = \{ g:= y^2 + x_0^2y + x_1^4 =0\} \subset 
\P(1, 1, 1, 2)_{[x_0: x_1: x_2: y]}. 
\]
By $X_{\overline k} \subset D_+(x_0) \cup D_+(x_1) \cup D_+(x_2)$, 
we may apply \autoref{r Jacobian}, and hence 
the non-smooth locus $(X_{\overline k})_{\text{non-reg}}$ of $X_{\overline k}$ is given by 
\[
(X_{\overline k})_{\text{non-reg}} = \{ g = \partial_{x_0}g = \partial_{x_1}g =\partial_{x_2}g = \partial_y g =0\} 
= \{ x_0 = y+x_1^2= 0\}, 
\]
which is one-dimensional. Hence (2) and (3) holds. 


{ 
Let us show (4). 
We have 
\[
X_{\overline k} \cap \{ x_1 =0\} = \{x_1 = y(y+tx_0^2)=0\}, 
\]
which is reducible. 
Hence $-K_{X_{\overline k}} \sim \MO_{\P(1, 1, 1, 2)}(1)|_{X_{\overline k}} \sim \Gamma_1 + \Gamma_2$ for 
distinct prime divisors $\Gamma_1, \Gamma_2$ on $X_{\overline k}$. 
Here $\Gamma_1 + \Gamma_2$ is an effective Cartier divisor on $X_{\overline k}$, whilst each $\Gamma_i$ is not necessarily Cartier.  
For 
the conductor $C \subset (X_{\overline k})^N$ of 
the normalisation $\nu : (X_{\overline k})^N \to X_{\overline k}$, we get 
\[
-K_{(X_{\overline k})^N} \sim C-\nu^*K_{X_{\overline k}} 
\sim C + \nu^*(\Gamma_{1}+\Gamma_{2}), 
\]
where $\nu^*(\Gamma_{1}+\Gamma_{2})$ is an effective Cartier divisor 
on $(X_{\overline k})^N$ 
which contains at least two prime divisors 
$\Gamma'_1$ and $\Gamma'_2$, 
where each $\Gamma'_i$ is a prime divisor contained in 
$\nu^{-1}(\Gamma_i)$. 
Therefore, we get 
\begin{equation}\label{e1 e 1-2}
    -K_{(X_{\overline k})^N} \sim C + \Gamma'_1 + \Gamma'_2 + (\text{an effective Weil divisor}). 
\end{equation}
Recall that either  $(X_{\overline k})^N \simeq \P^1 \times \P^1$ or 
$(X_{\overline k})^N \simeq  \P(1, 1, 2)$ (\autoref{thm: class_dP}). 
Then this linear equivalence (\ref{e1 e 1-2}) 
is impossible for the latter case $(X_{\overline k})^N \simeq \P(1, 1, 2)$ (because 
we have 
$-K_{\P(1, 1, 2)} \sim 2L$ and 
$\text{Cl}(\P(1, 1, 2)) = \Z L$
for some prime divisor $L$ on $\P(1, 1, 2)$,  
where $\text{Cl}(\P(1, 1, 2))$ denotes the divisor class group). 
Therefore, $(X_{\overline k})^N \simeq \P^1 \times \P^1$, 
i.e., (4) holds. 
}

Let us show (5). 
Since $X$ is a hypersurface in $\P(1, 1, 1, 2)$ of degree $4$, 
we have $K_X^2=2$. 
Moreover, we get $\rho(X)=1$ by \autoref{p pic surje} below. 
In what follows, let us check that \autoref{p pic surje} is actually applicable. Recall that the weighted projective space $Y := \P_{\F}(1, 1, 1, 2)$  over $\F$ is $\Q$-factorial. 
Let 
\[
\varphi \colon Y = \P_{\F}(1, 1, 1, 2) 
\to \P^4_{\F}
\]
be the $\F$-morphism induced by the base point free linear system $\Lambda$ on $\P_{\F}(1, 1, 1, 2)$ 
that is generated by $y^2, x_0^2y, x_0^4, x_1^4, x_2^4$. 
The universal family $Y^{{\rm univ}}_{\varphi}$ parametrising the members of $\Lambda$ is given by 
\[
Y^{{\rm univ}}_{\varphi} := \{ u_0 y^2 + u_1 x_0^2y +  u_2  x_0^4 + u_3  x_1^4 + u_4 x_2^4=0\} 
\hookrightarrow Y \times_{\F} \P^4_{\F, [u_0: \cdots : u_4]}. 
\]
Then the generic member $Y^{\gen}_{\varphi}$ of $\varphi$, 
in the sense of \cite[Definition 4.3]{Tan24}, is nothing but the generic fibre 
of the induced morphism $Y^{{\rm univ}}_{\varphi} \to \P^4_{\F, [u_0: \cdots : u_4]}$. Hence the following holds for $v_i :=u_i/u_0$: 
\[
Y^{\gen}_{\varphi} \simeq 
\{ y^2 + v_1 x_0^2y +  v_2  x_0^4 + v_3  x_1^4 + v_4 x_2^4=0\} 
\simeq X. 
\]

\end{ex}

\begin{prop}\label{p pic surje}
Let $\F$ be a field and let $Y$ be a projective normal $\Q$-factorial variety over $\F$. 
Let $\varphi : Y \to \P^n_{\F}$ be an $\F$-morphism, where $n$ is a non-negative integer. 
Then the map 
\[
\beta^* : \Pic Y \otimes_{\Z} \Q \to \Pic Y^{\gen}_{\varphi} \otimes_{\Z} \Q, \qquad L \mapsto \beta^*L
\]
is surjective, where 
$Y^{\gen}_{\varphi}$ is the generic member of $\varphi$  and 
$\beta : Y^{\gen}_{\varphi} \to Y$ denotes the induced morphism 
\cite[Definition 4.3]{Tan24}. 
\end{prop}

\begin{proof}
The assertion holds by applying the same argument as in  
\cite[Proposition 5.17]{Tan24}  
after replacing $\Pic (-)$ by $\Pic (-) \otimes_{\Z} \Q$. 
\end{proof}

\begin{ex}[$p=3, \rho(X)=1, K_X^2=3$]\label{e 1-3}
Let $\F$ be an algebraically closed field of characteristic three. 
Set $k :=\mathbb F(s, t)$ and 
\[
X:= \{ x^2y + xy^2 + sz^3 + tw^3 =0\} \subset \P^3_{[x: y: z: w]}. 
\]
In what follows, we prove the following. 
\begin{enumerate}
\item $X$ is regular. 
\item $X$ is geometrically integral.  
\item $X$ is not geometrically normal. 
\item $\rho(X)=1$ and $K_X^2=3$.
\end{enumerate}

Let us show (1). 
We apply the Jacobian criterion (Proposition \ref{p Jacobian}). 
The equation 
$(\partial_sf, \partial_tf)=(0, 0)$ implies $z=w=0$. 
This, together with $\partial_x f  =  \partial_y f=0$, 
implies $x^2y + xy^2 = 2xy + y^2 = x^2 +2xy=0$. 
In particular, we get $x \neq 0$ and $y \neq 0$, and hence 
 the equation
 $x^2y + xy^2 = 2xy + y^2 = x^2 +2xy=0$ is equivalent to 
 $x +y = 2x +y = x +2y=0$, which has no solution except for $(x, y) =(0, 0)$.  
 Thus (1) holds. 



Let us show (2) and (3). 
After taking the base change to $\overline k$, we get 
\[
X_{\overline k} 
=\{ g := x^2 y + xy^2 +z^3 =0 \} \subset \P^3_{[x: y: z: w]}. 
\]
By the Jacobian criterion (Proposition \ref{p Jacobian}), the singular locus $(X_{\overline k})_{\text{non-reg}}$ of $X_{\overline k}$ is given by 
\[
(X_{\overline k})_{\text{non-reg}} = \{  x^2 y + xy^2 +z^3 =2xy +y^2 =2xy +x^2=0\}. 
\]
Therefore, the singular locus $(X_{\overline k})_{\text{non-reg}}$ 
of $X_{\overline k}$ contains the line $\{x=y=z\}$, 
and hence $X_{\overline k}$ is not normal. 
Thus (3) holds.  
 In order to show (2), it is enough to
find a smooth point of $X_{\overline k}$. 
Pick $\alpha \in \overline k$ such that 
$\alpha^2 + 2\alpha \neq 0$. 
Then $X_{\overline k}$ is smooth at the point $[x:y:z:w] = [\alpha :1:\beta:1]$, 
where $\beta$ is the solution of the equations $x^2y +xy^2 + z^3 =0$ and $(x, y) = (\alpha, 1)$ (i.e., $\beta := (-(\alpha^2+\alpha))^{1/3}$). 
Thus (2) holds.  
Since $X$ is a cubic surface in $\P^3_k$, we get $K_X^2=3$. 
It follows from \autoref{thm: class_dP} that $\rho(X)=1$. 
Thus (4) holds.
\end{ex}


\begin{ex}
[$p=2, \rho(X)=1, K_X^2=4$]\label{e 1-4}
Let $\F$ be an algebraically closed field of characteristic two. 
Set $k :=\mathbb F(s_1, s_2, s_3, s_4, t_1, t_2, t_3, t_4)$ and  
\[
X:= \{ x_0x_1 + 
s_1x_1^2+s_2x_2^2 +s_3x_3^2 +s_4x_4^2   
=  x_0x_2+ 
t_1x_1^2+t_2x_2^2 +t_3x_3^2 +t_4x_4^2=0\} \subset 
\P^4_{[x_0: x_1: x_2: x_3: x_4]}. 
\]
In what follows, we prove the following.  
\begin{enumerate}
\item $X$ is regular. 
\item $X$ is geometrically integral. 
\item $X$ is not geometrically normal. 
\item  $(X_{\overline k})^N \simeq \P^2_{\overline k}$, $\rho(X)=1$, 
and $K_X^2=4$. 
\end{enumerate}

Let us show (1) and (2).  
We have 
\[
D_+(x_0)=\{ x_1 + s_1x_1^2+s_2x_2^2 +s_3x_3^2 +s_4x_4^2   
=  x_2+ t_1x_1^2+t_2x_2^2 +t_3x_3^2 +t_4x_4^2=0\} \subset 
\A^4_{(x_1, x_2, x_3, x_4)}. 
\]
By the Jacobian criterion, $D_+(x_0)$ is smooth over $k$. 
We have 
\[
D_+(x_1)=\{ x_0 + s_1+s_2x_2^2 +s_3x_3^2 +s_4x_4^2   
=  x_0x_2+ t_1 +t_2x_2^2 +t_3x_3^2 +t_4x_4^2=0\} \subset 
\A^4_{({x_0}, x_2, x_3, x_4)}. 
\]
By the Jacobian criterion using $\partial_{s_1}$ and 
 $\partial_{t_1}$, 
$D_+({x_1})$ is regular. 
We can check that {$D_+(x_2)$}, $D_+(x_3)$, and $D_+(x_4)$ are regular 
in a similar way. 
Thus (1) holds. 
As we have shown that $X$ has a smooth point, 
(2) holds. 

Let us show (3). 
We have 
\[
X_{\overline k} \simeq  
\{ x_0x_1 + x_3^2=  x_0x_2+x_4^2=0\} \subset 
\P^4_{[x_0: x_1: x_2:x_3: x_4]}. 
\]
We get 
\[
D_+(x_1) = \{ x_0 + x_3^2 = x_0x_2 + x_4^2 =0\} \subset \A^4_{(x_0, x_2, x_3, x_4)}. 
\]
We have 
\[
\Gamma(D_+(x_1), \MO_{X_{\overline k}}) \simeq \frac{\overline{k}[x_0, x_2, x_3, x_4]}{(x_0 + x_3^2, x_0x_2 + x_4^2)} \simeq 
\frac{\overline{k}[x_2, x_3, x_4]}{(x_2x_3^2 + x_4^2)}. 
\]
By Jacobian criterion, the singular locus of $\{ x_2x_3^2 +x_4^2 =0\} \subset \A^3_{(x_2, x_3, x_4)}$ 
contains the affine line $\{ x_3 = x_4=0\}$. 
Thus (3) holds. 

Let us show (4). 
We work over $\overline k$. 
Consider the morphism 
\[
f : \P^2_{[u:v:w]} \to \P^4_{[x_0 : x_1:x_2:x_3:x_4]}, 
\quad [u:v:w] \mapsto [u^2:v^2:w^2:uv:uw]. 
\]
It is clear that $f$ factors through $X_{\overline k} =   
\{ x_0x_1 + x_3^2=  x_0x_2+x_4^2=0\}$: 
\[
f: \P^2_{[u:v:w]} \xrightarrow{g} X_{\overline k} \hookrightarrow \P^4_{[x_0 : x_1:x_2:x_3:x_4]}. 
\]
In order to prove $(X_{\overline k})^N \simeq \P^2_{\overline k}$, it is enough to show that $g$ is birational. 
This follows from 
$(\MO_{\P^4}(1)|_{X_{\overline k}})^2 = \deg X_{\overline k} = 4$ 
and $(f^*\MO_{\P^4}(1))^2 = (\MO_{\P^2}(2))^2=4$. 
Hence ${(X_{\overline k})^N} \simeq \P^2_{\overline k}$. 
In particular, $\rho(X)=1$. 
Since $X$ is a complete intersection of two quadrics, we get $K_X^2=4$. 
Thus (4) holds.






\end{ex}





\begin{ex}[$p=2, \rho(X)=2, K_X^2=2$]\label{e 2-2} 
Set 
\[
k := \F(s_0, s_1, s_2, t_0, t_1, t_2, u_0, u_1, u_2), 
\]
where $\F$ is  an algebraically closed field of characteristic two. 
Let $X \subset \P(1, 1, 1, 2)_{[x_0:x_1:x_2:y]}$ be the weighted hypersurface of degree $4$ given  by 
\[
X := \{ f:= y^2  + f_2 y + f_4=0 \} \subset \P(1, 1, 1, 2)_{[x_0:x_1:x_2:y]}
\]
\[
\text{where}\qquad f_2 := \sum_{i=0}^2 s_ix^2_i \qquad \text{and}\qquad 
f_4 := \left(\sum_{i=0}^2 t_ix_i^2\right)\left(\sum_{i=0}^2 u_ix_i^2\right). 
\]
In particular, each $f_i \in k[x_0, x_1, x_2]$ is a homogeneous polynomial of degree $i$. 

In what follows, we prove the following. 
\begin{enumerate}
\item $X$ is regular. 
\item $X$ is geometrically integral. 
\item $X$ is not geometrically normal.  
\item $K_X^2 =2$ and $\rho(X) = \rho(X \times_k k^{\sep})=2$. 
\end{enumerate}
Recall that $[0:0:0:1] \in \P(1, 1, 1, 2)$ is the unique non-smooth point of 
$\P(1, 1, 1, 2)$. 
Since the defining equation $f$ satisfies $f(0, 0, 0, 1) \neq 0$, 
we have that $X \subset \bigcup_{i=0}^2 D_+(x_i)$. 
In order to show (1), it is enough,  by symmetry of the defining  equation $f$ of $X$ with respect to the variables $x_0, x_1, x_2$, to prove that $X \cap D_+(x_0)$ is regular. 
It holds that 
\[
X \cap D_+(x_0) \simeq \{ y^2 + \wt{f}_2y + \wt{f}_4 =0\} \subset \mathbb A^3_{(x_1, x_2, y)},
\]
\[
\text{where}\qquad \wt{f}_i:= \wt{f}_i(x_1, x_2) := f_i(1, x_1, x_2) \in k[x_1, x_2]. 
\]
By Jacobian criterion using the derivations $\partial_y, \partial_{u_0}, \partial_{t_0}$, 
the non-regular locus of $X \cap D_+(x_0)$ is contained in 
\[
\{   s_0 + s_1 x^2_1 + s_2 x_2^2 =   t_0 + t_1 x^2_1 + t_2 x_2^2 =  u_0 + u_1 x^2_1 + u_2 x_2^2 = 0\},  
\]
which is empty. 
Hence (1) holds. 
In order to show (2) and (3), we apply the Jacobian criterion to $X_{\overline k} \cap D_+({x_0})$. The Jacobian matrix is given by 
\[
(\partial_{x_1}f, \partial_{x_2}f, \partial_{y}f) =  (0, 0, \wt{f}_2). 
\]
Hence the non-smooth locus is given by 
$\{ \wt{f}_2 = y^2 + \wt{f}_4 =0\} \subset \A^3_{(x_1, x_2, y)}$. 
This locus is (at least) one-dimensional, and hence $X$ is not geometrically normal. 
Moreover, $X$ is geometrically integral, as 
$X$ has  a smooth point.  
Hence (2) and (3) hold. 

Let us show (4). 
Recall that $X$ is an effective Cartier divisor on $\P(1, 1, 1, 2)$ which does not pass through the unique singular point of $\P(1, 1, 1, 2)$. 
By the adjunction formula, 
we get $K_X^2 = 2$. 
It is enough to show $\rho (X) \geq 2$. 
Since $X$ is regular, it suffices to find two one-dimensional closed subschemes $C_1$ and $C_2$ on $X$ satisfying $C_1 \cap C_2 = \emptyset$. 
Set 
\begin{eqnarray*}
C_1 &:=& \{ y  + \sum_{i=0}^2 s_ix_i^2 =  \sum_{i=0}^2 t_ix_i^2 =    0\}\\
C_2 &:=& \{ y  + \sum_{i=0}^2 t_ix_i^2 = y + \sum_{i=0}^2 (s_i+u_i)x_i^2  =    0\}. 
\end{eqnarray*}
It is easy to see that $C_1 \subset X$ and $C_2 \subset X$. 
Then the problem is reduced to showing 
\begin{enumerate}
\item[(i)] $C_1 \cap C_2 = \emptyset$, and 
\item[(ii)] $C_1$ and $C_2$ are one-dimensional.
\end{enumerate}
The assertion (i) follows from 
\begin{eqnarray*}
C_1 \cap C_2 
&=& \{ y  + \sum_{i=0}^2 s_ix_i^2 =  \sum_{i=0}^2 t_ix_i^2 =    0\} 
\cap \{ y  + \sum_{i=0}^2 t_ix_i^2 = y + \sum_{i=0}^2 (s_i+u_i)x_i^2  =    0\}\\
&=& \{ y = \sum_{i=0}^2 s_ix_i^2 =  \sum_{i=0}^2 t_ix_i^2 = 
\sum_{i=0}^2 u_ix_i^2  =    0\} \\
&=& \emptyset, 
\end{eqnarray*}
where the last equality is assured by  the fact that 
the intersection of the corresponding three lines  on $\P^2_{\overline k}$  is empty: 
\[
\{  \sum_{i=0}^2 \sqrt{s_i}x_i = \sum_{i=0}^2 \sqrt{t_i}x_i =\sum_{i=0}^2 \sqrt{u_i}x_i = 0\}  = \emptyset \qquad \text{in}\qquad \P^2_{\overline k} = \P^2_{[x_0:x_1:x_2]}. 
\]
Let us show (ii). 
By symmetry, we only check that $\wt{C}_i := C_i \cap D_+(x_0)$ is one-dimensional for each $i \in \{1, 2\}$. We see that 
$\wt{C}_1$ is one-dimensional by  
\[
\wt{C}_1 = \Spec\, \frac{k[x_1, x_2, y]}{(y + s_0 + s_1x_1^2 +s_2x_2, t_0+t_1x_1^2 +t_2x_2^2)} 
\simeq \Spec\, \frac{k[x_1, x_2]}{(t_0+t_1x_1^2 +t_2x_2^2)}. 
\]
Similarly, $\wt{C}_2$ is one-dimensional. 
This completes the proofs of (ii) and (4). 
\end{ex}

\begin{ex}
[$p=2, \rho(X)=2, K_X^2=3$]\label{e 2-3}
Let $\F$ be an algebraically closed field of characteristic two. 
Set $k :=\mathbb F(s_1, s_2, s_3, s_4, t_1, t_2, t_3, t_4)$ and let 
\[
Y:= \{ x_0x_1 + 
s_1x_1^2+s_2x_2^2 +s_3x_3^2 +s_4x_4^2   
=  x_0x_2+ 
t_1x_1^2+t_2x_2^2 +t_3x_3^2 +t_4x_4^2=0\} \subset 
\P^4_{[x_0: x_1: x_2: x_3: x_4]} 
\]
be the example of \autoref{e 1-4}. 
For $P := [1:0:0:0:0] \in Y(k)$, let $\sigma : X \to Y$ be the blowup at $P$. 
By \autoref{e 1-4}, 
$X$ is regular, geometrically integral and geometrically non-normal. 
By \autoref{l 1-4}{(2)}, $-K_X$ is ample.


\end{ex}

\begin{ex}
[$p=2, \rho(X)=2, K_X^2=4$]\label{e 2-4}
Set $k := \F(t_1, t_2, t_3, t_4)$, 
where $\F$ is  an algebraically closed field of characteristic two. 
Let $\pi : X \to \P^1_{[x:y]} \times \P^1_{[x':y']}$ be the double cover given  by 
\[
X := \{ w^2  + xyx'^2+ t_1x^2x'^2 + t_2 y^2x'^2 + t_3 x^2y'^2 +t_4 y^2 y'^2=0 \} \subset \Spec_Y \bigoplus_{m=0}^{\infty}\mathcal L^{\otimes -m}, 
\]
which is defined by the same way as the cyclic cover \cite[Definition 2.50]{km-book}. 
Specifically, for $Y :=  \P^1_{[x:y]} \times \P^1_{[x':y']}$, 
$\mathcal L := \MO_{\P^1 \times \P^1}(1, 1)$, 
and 
\[
 s:= xyx'^2+ t_1x^2x'^2 + t_2 y^2x'^2 + t_3 x^2y'^2 +t_4 y^2 y'^2
 \in H^0(Y, \mathcal L^{\otimes 2}),
\]
$X$ is defined by gluing as follows: 
\[
X = \bigcup_{i \in I} \Spec \Gamma(Y_i, \MO_{Y_i})[w_i]/(w_i^2 + s_i)
\]
where $Y = \bigcup_{i \in I} Y_i$ is an open cover trivialising $\mathcal L$, and $s_i \in \Gamma(Y_i, \MO_{Y_i})$ is the element corresponding to $s|_{Y_i}$ via a fixed trivialisation $\mathcal L|_{Y_i} \xrightarrow{\simeq} \MO_{Y_i}$. 

Then we have 
\[
D_+(y) \cap D_+(y') = 
\{ f:= w^2  + xx'^2+ t_1x^2x'^2 + t_2 x'^2 + t_3 x^2 +t_4 =0 \} \subset \A_{(x, x', w)}^3. 
\]
This is regular by the Jacobian criterion 
using $\partial_{t_4}$. 
Similarly, we can check that $X$ is regular. 
Over $\overline k$, we have 
\[
\partial_xf  =\partial_{x'}f = \partial_wf= 0 
\Leftrightarrow x'^2 =0. 
\]
By the Jacobian criterion for smoothness, 
$X$ has a smooth point and the affine curve 
\[
x' = w^2  +  t_3 x^2 +t_4 =0
\]
is contained in the non-smooth locus of $X_{\overline k}$. 
Therefore, $X$ is geometrically integral but not geometrically normal. 
Since $\pi : X \to \P^1_k \times_k \P^1_k$ 
is a double cover given by 
\[
 xyx'^2+ t_1x^2x'^2 + t_2 y^2x'^2 + t_3 x^2y'^2 +t_4 y^2 y'^2
 \in H^0(\P^1_k \times_k \P^1_k, \MO_{\P^1 \times \P^1}(2, 2)), 
\]
we obtain $K_X \sim \pi^*(K_{\P^1 \times \P^1} + \MO_{\P^1 \times \P^1}(1, 1))$. 
Hence $-K_X$ is ample and 
\[
K_X^2 = (\pi^*\MO_{\P^1 \times \P^1}(-1, -1))^2 
= (\deg \pi) (\MO_{\P^1 \times \P^1}(-1, -1))^2 = 2 \cdot 2 = 4. 
\]
Finally, we have $\rho(X) \geq \rho(\P^1 \times \P^1)= 2$, 
which implies $\rho(X)=2$ (\autoref{thm: class_dP}). 
\end{ex}

\begin{example}
[$p=2, \rho(X)=2, K_X^2=5$]\label{e 2-5}
Let $\F$ be an algebraically closed field of characteristic two. 
Set $k:=\F(s, t)$. 
Take the closed point $P \in \A^2_{(x, y)} =D_+(z) \subset \P^2_{[x:y:z]}$ 
given by 
\[
P := \Spec k[x, y]/(x^2+s, y^2+t) \subset \Spec k[x, y] = \A^2_{(x, y)}. 
\]
Let $\sigma : X \to \P^2_{[x:y:z]}$ be the blowup at $P$. 
Then $X$ is a geometrically non-normal geometrically integral regular del Pezzo surface satisfying $K_X^2 =5$ and $\rho(X)=2$ 
\cite[Example 6.23]{BFSZ}. 
Since $X$ is a pencil generated by the conics $\{x^2 +sz^2 =0\}$ and $\{ y^2+tz^2=0\}$, 
the equation of $X$ is explicitly given as follows: 
\[
X = \{ u(x^2 +sz^2) + v(y^2 +tz^2) =0\} \subset \P^2_{[x:y:z]} \times 
\P^1_{[u:v]}. 
\]
\end{example}

\begin{example}
[$p=2, \rho(X)=2, K_X^2=6$]\label{e 2-6}
Let $\F$ be an algebraically closed field of characteristic two. 
Set $k:=\F(s)$ and 
\[
Y := \Proj k[x, y, z, w]/(x^2+sy^2+zw). 
\] 
Take the closed point 
\[
Q \in D_+(y) \simeq \Spec k[x, z, w]/(x^2+s+zw) \subset Y 
\]
defined by the maximal ideal $\m:=(x^2+s, z, w)$ of $k[x, z, w]/(x^2+s+zw)$. 
Let $f:X \to Y$ be the blowup at $Q$. 
Then the following hold \cite[Example 6.5]{Tan19} 
(cf.~\cite[Example 6.25]{BFSZ}). 
\begin{enumerate}
\item  $Y$ is a regular projective surface such that 
$H^0(Y, \MO_Y)=k$, $K_Y^2=8$, and 
$Y \times_k \overline k \simeq \mathbb P(1, 1, 2)$. 
\item $-K_X$ is ample and $K_X^2=K_Y^2-2=6$. 
\item $X$ is a regular del Pezzo surface which is geometrically integral but not geometrically normal. 
\end{enumerate}
\end{example}

\begin{ex}[$p=2, \rho(X)=1, K_X^2=1, h^1(X, \MO_X) =1$]\label{e 1-1-i}
Let $\F$ be an algebraically closed field of characteristic two. 
For $k:=\F(s_0, s_1, s_2, s_3)$, we set 
\[
Z := \Proj \frac{k[x_0, x_1, x_2, x_3]}{(s_0x_0^2+ s_1x_1^2 + s_2x_2^2 +s_3x_3^2)} \subset  \Proj k[x_0, x_1, x_2, x_3] = \P^3_k. 
\]
Then, for a suitable foliation $\mathcal F$ on $Z$, 
the quotient $X := Z/\mathcal F$ is a geometrically integral regular del Pezzo surface over $k$ such that 
$\rho(X) =1$, $K_X^2=1$, and 
$h^1(X, \MO_X) =1$. 
For more details, see \cite[Subsection 3.1 and Subsection 3.2]{Mad16}. 
\end{ex}


\section{Arithmetic of regular del Pezzo surfaces}

Let $X$ be a geometrically integral regular del Pezzo surface. 
The main purpose of this section is to establish the following two results. 
\begin{enumerate}
\item If $X(k) \neq \emptyset$ and $K_X^2 \geq 5$, then $X$ is rational (rationality criterion,  
\autoref{t rationality}). 
\item If $k$ is a $C_1$-field, then $X(k) \neq \emptyset$ 
(the $C_1$-conjecture for regular del Pezzo surfaces, \autoref{thm: existence_section}). 
\end{enumerate}
Following the strategy for the smooth case \cite{Man66}, 
we divide the proofs of (1) and (2) 
according to the degree $K_X^2$. 
Finally, we discuss unirationality of del Pezzo surfaces of degree 4 and 3.


\subsection{Degree $>6$}

We recall how to compute the Picard group of varieties over fields with trivial Brauer group.

\begin{lemma} \label{lem: picard-invariant}
    Let $X$ be a projective variety {with $H^0(X, \mathcal{O}_X)=k$}.
    If $\Br(k)=0$, then $\Pic(X)=\Pic(X_{k^{\sep}})^{\Gal(k^{\sep}/k)}$.
\end{lemma}

By \cite[Chapitre II, Proposition 5 and 8]{ser-coh-gal},  
a $C_1$-field $k$ has trivial Brauer group $\Br(k)=0$.

\begin{proof}
See \cite[Proposition 5.4.2]{CTS21}. 
\end{proof}

\begin{prop}[degree 9]\label{p deg9}
Let $X$ be a geometrically integral regular del Pezzo surface with $K_X^2=9$.
Then the following hold. 
\begin{enumerate}
\item $X(k) \neq \emptyset$ if and only if $X \simeq \P^2_k$. 
\item If $\Br(k)=0$ (e.g., $k$ is a $C_1$-field), then $X \simeq \P^2_k$. 
\end{enumerate}
\end{prop}

\begin{proof}
{By \cite[Lemma 4.6]{BFSZ}, $X$ is a Severi--Brauer variety and hence (1) follows}. 
Let us show (2). 
Assume that $\Br(k)=0$. 
By the short exact sequence of sheaves on the \'etale site of $\Spec(k)$:
$$0 \to \mathbb{G}_m \to \GL_{3} \to \PGL_{3} \to 0,$$ 
we deduce an exact sequence of pointed sets:
$$H^1_{\et}(\Spec(k), \GL_{3}) \to H^1_{\et}(\Spec(k),\PGL_{3}) \to \Br(k)=0. $$
As $H^1_{\et}(\Spec(k), \GL_{3})=0$ by Hilbert's 90, we conclude that $H^1_{\et}(\Spec(k),\PGL_{3})=0$ and thus we deduce $X$ is isomorphic to $\mathbb{P}^2_k$.
\qedhere



\end{proof}

\begin{prop}[degree 8]\label{p deg8}
Let $X$ be a geometrically integral regular del Pezzo surface with $K_X^2=8$. 
Then the following hold. 
\begin{enumerate}
\item If $X(k) \neq \emptyset$, then $X$ is rational. 
\item If $k$ is a $C_1$-field, then $X$ is rational. 
\end{enumerate}
\end{prop}

\begin{proof}
By \autoref{thm: class_dP}, $X$ is geometrically normal  and 
hence $X$ is geometrically canonical. 
If $X \simeq \F_1 := 
{\mathbb{P}_{\mathbb{P}^1}(\mathcal{O}_{\P^1} \oplus \mathcal{O}_{\P^1}(1))}$, 
then there is nothing to show. 
In what follows, we assume $X \not\simeq \F_1$.

Let us show (1). 
Assume $X(k) \neq \emptyset$. 
It follows from  \cite{BFSZ}*{Proposition 4.10} that 
$-K_X \sim 2H$ for some ample Cartier divisor $H$. 
By \cite{BFSZ}*{Lemma 4.9}, $X$ is isomorphic to a quadric surface in $\P^3_k$. 
A regular quadric surface $X \subset \P^3_k$ 
with a $k$-rational point $P$ is birational to $\mathbb{P}^2_k$ by the classical argument of the projection $\pi_P \colon X \dashrightarrow \mathbb{P}^2_k$ from the point $P$. 
Thus (1) holds. 

Let us show (2). 
Assume that $k$ is a $C_1$-field. 
Since $X$ is geometrically integral,  
the base change $X_{k^{\sep}}$ has a $k^{\sep}$-rational point. 
By (1), $X_{k^{\sep}}$ is a quadric surface in $\mathbb{P}^3_{k^{\sep}}$. 
For a hyperplane section $H'$ on $X_{k^{\sep}}$, 
the adjunction formula implies $-K_{X_{k^{\sep}}} \sim 2H'$. 
Recall that $\Pic X$ is torsion free. 
Then the isomorphism class of $\MO_{X_{k^{\sep}}}(H')$ is fixed by the Galois group, because so is that of the dualising sheaf $\MO_{X_{k^{\sep}}}  (K_{X_{k^{\sep}}})$. 
By \autoref{lem: picard-invariant}, 
there exists a Cartier divisor $H$ on $X$ such that $-K_X\sim 2H.$
By \cite{BFSZ}*{Lemma 4.9}, we conclude that $X$ is a quadric surface in $\mathbb{P}^3_k$. 
Then $X$ has a $k$-rational point 
as the base field $k$ is $C_1$. 
Hence (1) implies that $X$ is rational. 
Thus (2) holds. 
\qedhere

\end{proof}

\begin{prop}[degree 7]\label{p deg7}
Let $X$ be a geometrically integral regular del Pezzo surface with $K_X^2=7$.
Then there exists a birational morphism $X \to \P^2_k$. 
In particular, $X$ is rational. 
\end{prop}

\begin{proof} The assertion follows from \cite{BFSZ}*{Lemma 4.12}. \qedhere  \end{proof}

\begin{cor}\label{c deg789}
Let $X$ be a geometrically integral regular del Pezzo surface with $K_X^2 \in \{7, 8, 9\}$.
Assume that $X(k) \neq \emptyset$ or 
$k$ is a $C_1$-field. 
Then $X$ is rational. 
\end{cor}
\begin{proof}
    The assertion immediately follows from 
     \autoref{p deg9},
     \autoref{p deg8}, and 
     \autoref{p deg7}.
\end{proof}

\subsection{Rationality criteria}


In this subsection, we establish two rationality criteria. 
Specifically, we prove that $X$ is rational if $X$ is 
a geometrically integral projective regular surface over a field $k$ 
such that $K_X^2 \geq 5$, $X(k) \neq \emptyset$, 
and one of the following holds. 
\begin{itemize}
\item $-K_X$ is ample (\autoref{t rationality}). 
\item $X$ has a Mori fibre space structure $\pi : X \to B$ (\autoref{thm: rat_min_MFS}). 
\end{itemize}

    

\begin{thm} 
\label{t rationality}
    Let $X$ be a geometrically integral regular del Pezzo surface 
    with $K_X^2 \geq 5$ and $X(k) \neq \emptyset.$ 
    Then $X$ is rational. 
\end{thm}

\begin{proof}
We prove the assertion 
by descending induction on $K_X^2$. 

\setcounter{step}{0}

\begin{step}\label{s1 rationality}
$X$ is rational if $K_X^2 \geq 7$. 
\end{step}

\begin{s1proof}
It follows from \cite[Theorem 4.7]{Tan19} that $K_X^2 \leq 9$. 
Hence the assertion follows from 
 \autoref{c deg789}. 
This completes the proof of Step \ref{s1 rationality}. 
\end{s1proof}

\medskip

\noindent 
In what follows, we assume $K_X^2 \in \{ 5, 6\}$ and 
the induction hypothesis $(\star)$ below. 
\begin{enumerate}
\item[$(\star)$] 
$W$ is rational if $W$ is a regular del Pezzo surface with $K_W^2 >K_X^2$ and $W(k) \neq \emptyset$. 
\end{enumerate}
It is enough to show that $X$ is rational.



\begin{step}\label{s2 rationality}
If $\rho(X) \geq 2$, then $X$ is rational. 
\end{step}

\begin{s2proof}
Assume $\rho(X) \geq 2$. 
By $K_X^2 \in \{5, 6\}$ and \cite[Lemma 4.24]{BFSZ}, 
there exists an extremal ray which induces a birational morphism $X \to Y$. 
Then $Y$ is a geometrically integral regular del Pezzo surface 
with $K_X^2 < K_Y^2$ and $Y(k) \neq \emptyset$ (cf. \cite[Lemma 6.11]{BT22}). 
By the induction hypothesis $(\star)$, 
 $Y$ is rational, and hence so is $X$. 
This completes the proof of Step \ref{s2 rationality}. 
\end{s2proof}

\begin{step}\label{s3 rationality}
If $\rho(X) =1$, then $X$ is rational. 
\end{step}

\begin{s3proof}
By $K_X^2 \in \{5, 6\}$, we get $\Pic X = \Z K_X$ 
 and $H^1(X, \MO_X)=0$ (\autoref{thm: class_dP}). 
Fix a $k$-rational point $P$ on $X$. 
Let $\sigma : Y \to X$ be the blowup at $P$. 
It follows from  \autoref{l rho=1} that $Y$ is a regular del Pezzo surface. 
By \cite[Proposition 4.32 and Proposition 4.35]{BFSZ}, 
the following hold for the contraction $\tau : Y \to Z$ of the extremal ray not corresponding to $\sigma$. 
\begin{itemize}
\item $\tau : Y \to Z$ is a birational morphism. 
\item $Z$ is a geometrically integral regular del Pezzo surface. 
\item $K_X^2 < K_Z^2$. 
\end{itemize}
It follows from the Lang-Nishimura theorem (cf. \cite[Lemma 6.11]{BT22}) that 
$Z(k) \neq \emptyset$. 
By the induction hypothesis $(\star)$, 
$Z$ is rational, and hence so is $X$. 
This completes the proof of Step \ref{s3 rationality}. 
\end{s3proof}

\medskip

Step \ref{s2 rationality} and Step \ref{s3 rationality} 
complete the proof of \autoref{t rationality}. 
\qedhere









    
\end{proof}

\begin{lem} \label{lem: conic_bdl}
Let $X$ be a geometrically integral regular projective surface with a Mori fibre space $\pi \colon X \to B$ onto a regular projective curve $B$. 
Suppose $B(k) \neq \emptyset$.
Assume that $\pi$ does not have a section. 
Then the following hold. 
\begin{enumerate}
\item  $N^1(X) = \Z K_X \oplus \Z F$ for a fibre $F$ of $\pi$ over a $k$-rational point of $B$. 
\item If $K_X^2 >0$, then $-K_X$ is big. 
\item $K_X^2 \leq 4$ or $-K_X$ is ample. 
\end{enumerate}
\end{lem}

\begin{proof}
Let us show (1). 
Pick a Cartier divisor $D$ on $X$. 
It is enough to show that $D \equiv  a K_Z + b F$ for some $a, b \in \Z$. 
Since 
{there} is no section of $\pi$, 
we get $\Pic X_{\eta} = \Z K_{X_{\eta}}$ for the generic fibre $X_{\eta}$ of $\pi$ 
\cite[Lemma 10.6]{kk-singbook}. 
In particular, we have $D|_{X_{\eta}} \sim aK_X|_{X_{\eta}}$ for some $a \in \Z$, which implies $(D -aK_X) \cdot F=0$. 
We then get $D -a K_X \equiv bF$ for some $b \in \Z$ by the following exact sequence \cite[Theorem 4.4]{Tan18}: 
\[
0 \to N^1(B) \to N^1(X) \xrightarrow{\cdot F} \Z. 
\]
Thus (1) holds. 

Let us show (2). 
As $-K_X$ is $\pi$-ample, we deduce that 
\[
h^2(X, \mathcal{O}_X(-mK_X))=h^0(X, \mathcal{O}_X((m+1)K_X))=0
\]
for every $m>0$. 
By the Riemann--Roch theorem, 
we have 
$$h^0(X, \MO_X(-mK_X)) \geq \chi(X, \MO_X(-mK_X)) = 
\chi(X, \MO_X) + \frac{m(m+1)}{2}K_X^2.$$
Hence $-K_X$ is big when $K_X^2 >0$. 
Thus (2) holds. 

\medskip

Let us show (3). 
Suppose $K_X^2 \geq 5$ and $-K_X$ is not ample. 
Let us derive a contradiction. 

\begin{claim-}
Let $R$ be the extremal ray of $\overline{\NE}(X)$ not corresponding to $\pi$. 
Then the following hold. 
\begin{enumerate}
\item[(a)] $R = \R_{\geq 0}[C]$ for a prime divisor $C$ on $X$ satisfying $C^2 <0$.  
\item[(b)] $C \equiv - aK_X - b F$ for some integers $a>0$ and $b>0$. 
\end{enumerate}
\end{claim-}

\begin{claimproof}


Let us show (a). 
By (2), we can write  $-K_X \equiv A + E$, where $A$ is an ample $\mathbb{Q}$-divisor and $E$ is an effective $\mathbb{Q}$-divisor. 
In particular, $K_X+E+ (1/2)A \equiv -1/2 A$ is anti-ample. 
By \cite[Theorem 7.5(2)]{Tan18b}, 
we have $R = \R_{\geq 0}[C]$ for some prime divisor $C$ on $X$. 
If $C^2 \geq 0$, then we get $-K_X \cdot C = (A +E) \cdot C \geq A \cdot C >0$, and hence the Kleiman's criterion implies that $-K_X$ is ample, which is absurd. 
Therefore, $C^2 <0$. 
Thus (a) holds. 

Let us show (b). 
By (1), we have 
\[
C \equiv -aK_X -bF  
\]
for some $a, b \in \Z$. 
It suffices to show $a>0$ and $b>0$. 
We have $2a = (-aK_X -bF) \cdot F = C \cdot F >0$, 
where the inequality $C \cdot F >0$ is guaranteed by the fact that 
$C$ is a curve generating the {extremal} ray not corresponding to $\pi$. 
We then get $-K_X \equiv \frac{1}{a}C +\frac{b}{a}F$. 
It follows from $\kappa(-K_X) =2$, $\kappa(C) <2$, and $\kappa(F) <2$ 
that both coefficients $\frac{1}{a}$ and $\frac{b}{a}$ are positive, 
and hence  $b>0$. 
Thus (b) holds. 
This completes the proof of Claim. 
\end{claimproof}

\medskip

It is easy to see that the following hold. 
\begin{enumerate}
\item[(i)] $2a = C \cdot F = \deg (\pi|_C : C \to B) = [K(C):K(B)]$. 
\item[(ii)] 
$0 \leq K_X \cdot C = K_X \cdot (-aK_X -bF)  = -a K_X^2 +2b$. In particular, $aK_X^2 \leq 2b$.  
\item[(iii)] $C^2 = (-aK_X -bF)^2 =  a^2K_X^2 -4ab$. 
\end{enumerate}
We have that 
\begin{eqnarray*}
(K_X + C) \cdot C 
&\overset{{\rm (ii)(iii)}}{=}& (-a K_X^2 +2b)+ (a^2K_X^2  -4ab) \\
&=&(a^2-a)K_X^2 +2b(1-2a)\\
&\overset{{\rm (b)(ii)}}{\leq}& (a^2-a)K_X^2 + aK_X^2 (1-2a)\\
&=& aK_X^2 ( (a-1) +(1-2a))\\
&=& -a^2 K_X^2 \overset{{\rm (b)}}{<}0.
\end{eqnarray*}
For 
$k_C \coloneqq H^0(C, \MO_C)$, we get $-2[k_C:k]= (K_X+C) \cdot C \leq  -a^2 K_X^2$, and hence  
\[
[k_C:k] \geq  \frac{a^2K_X^2}{2}. 
\]
Since $B$ is geometrically integral, 
the ring $K(B) \otimes_k k_C$ is an integral domain. 
Then $K(B) \otimes_k k_C$
is a field, because 
$K(B) \otimes_k k_C$ is an integral domain which is integral over 
a field $K(B)$. 
Therefore, we get field extensions 
\[
K(B) \hookrightarrow 
K(B) \otimes_k k_C \hookrightarrow K(C), 
\]
which imply 
\[
 [K(C):K(B)] \geq [K(B) \otimes_k k_C :K(B)] = [k_C:k]. 
\]
To summarise, it holds that 
\[
\frac{a^2K_X^2}{2} \leq [k_C:k] \leq [K(C) : K(B)]
\overset{{\rm (i)}}{=} 2a. 
\]
We then get  $aK_X^2 \leq 4$, which contradicts $a \geq 1$ and $K_X^2 \geq 5$. Thus  (3) holds. 
\end{proof}

\begin{lemma} \label{lem: selfinter}
    Let $C$ be a geometrically integral regular projective curve and let $\mathcal{E}$ be a locally free sheaf of rank 2 on $C$.
    If $\pi \colon X=\mathbb{P}_{C}(\mathcal{E}) \to C$ is the projectivisation of $\mathcal{E}$, we have $K_X^2=8(1-h^1(C, \MO_C)).$
\end{lemma}

  This result is well known (see for example \cite[Chapter V, Corollary 2.11]{Ha77}). 
  However, we include the proof for sake of completeness as we work over an arbitrary field.

\begin{proof}
    Without any loss of generality, we can suppose that $k$ is separably closed.
    Let $S$ be a 
    section of $\pi$ and $F=\pi^*P$ be the fibre over a $k$-rational point $P$ of $C$.
    As 
    {$N^1(X)=\mathbb{Z}S \oplus \mathbb{Z}F$}, we can write 
    $K_X \equiv a S +b F$ for some $a, b \in \mathbb{Z}$.
    As $(K_X+F) \cdot F=-2$, we deduce $a=-2$.
    Therefore, we have the following chain of equalities:
    $$2h^1(C, \MO_C)-2=\deg_S(K_S)=(K_X+S) \cdot S=(-S+bF) \cdot S=-S^2+b,$$
    which implies that $b=2h^1(\MO_C)-2+S^2$.
    Thus we deduce 
    $$ K_X^2 =  (-2S+bF)^2 = 4S^2+2 \cdot (-2)\cdot (2h^1(\MO_C)-2+S^2)=8(1-h^1(C, \MO_C)), $$
    concluding.
\end{proof}

In \cite[Theorem 4.41]{BFSZ}, it is shown that the degree $K_X^2$ of a rational regular del Pezzo surface of Picard rank 1 satisfies the inequality $K_X^2 \geq 5$.
We prove a converse of this implication, 
generalising the work of Manin \cite{Man66}.

\begin{theorem} \label{thm: rat_min_MFS}
Let $X$ be a geometrically integral projective regular surface with a Mori fibre space structure $\pi \colon X \to B$. 
Then the following are equivalent. 
\begin{enumerate}
\item $X$ is rational. 
\item $K_X^2 \geq 5$ and  $X(k) \neq \emptyset$. 
\end{enumerate}
\end{theorem}


\begin{proof}
The implication (1) $\Rightarrow$ (2) has been settled in 
\cite[Theorem 4.41]{BFSZ}. 
In what follows, we prove the opposite implication (1) $\Leftarrow$ (2). 
Assume (2).     Let us show (1). 
    If $\dim B=0$, then this is proven in \autoref{t rationality}. 
    We may assume that $\dim B=1$, i.e., 
    $\pi : X \to B$ is a conic bundle. 

Assume that  $\pi$ has no section. 
In this case, $-K_X$ is ample (\autoref{lem: conic_bdl}(3)), and hence $X$ is rational (\autoref{t rationality}). 
Hence we may assume that $\pi$ has a section.     
    We then get $X \simeq \P_B(\mathcal E)$ for some locally free sheaf $\mathcal E$ on $B$ of rank $2$ \cite[Lemma 4.20]{BFSZ} 
    (i.e., $\pi: X \to B$ is a $\P^1$-bundle). 
    As $K_X^2=8(1-h^1(B, \MO_B))$ by \autoref{lem: selfinter}, we conclude that $K_X^2=8$ and $h^1(B, \MO_B)=0$. As $B(k) \neq \emptyset$, we deduce $B \simeq \mathbb{P}^1_k$ and 
    $X$ is a $\P^1$-bundle over $\P^1_k$, and thus rational.
\end{proof}

\begin{corollary} 
Let $X$ be a geometrically integral projective regular surface 
such that $\kappa(X)=-\infty$. 
If $K_X^2 \geq 5$ and $X(k) \neq \emptyset$, then $X$ is rational. 
\end{corollary}

\begin{proof}
    Running a $K_X$-MMP, we have a birational contraction $X \to Y$ of regular projective surfaces, 
    where $Y$ denotes the end result of this MMP \cite[Theorem 1.1]{Tan18}. 
    By $\kappa(Y)=\kappa(X) = -\infty$, $K_Y$ is not nef \cite[Theorem 1.1]{Tan20}. 
Then $Y$ has a Mori fibre space structure $Y \to Z$ \cite[Theorem 1.1]{Tan18}. 
    As $K_Y^2 \geq K_X^2 \geq 5$ and $Y(k) \neq \emptyset$, we conclude by \autoref{thm: rat_min_MFS} that $Y$ is rational, and thus $X$ also.
\end{proof}

\subsection{Degree 6}





In this subsection, we prove the existence of a rational point on a del Pezzo of degree 6 over a $C_1$-field.

\begin{lem}\label{l 0 dim lin sub}
Let $Z$ be a zero-dimensional closed subscheme of $\P^n_k$ 
such that $d:=\dim_k \Gamma(Z, \MO_Z) \leq  n$.  
Then $Z$ is contained in  a $(d-1)$-dimensional linear subvariety $V$ of $\P^n_k$. 
\end{lem}

\begin{proof}
Consider the exact sequence 
\[
0 \to H^0(\P^n_k, \MO_{\P^n}(1) \otimes \mathcal I_Z) 
\to  H^0(\P^n_k, \MO_{\P^n_k}(1)) 
\to  H^0(Z, \MO_{\P^n_k}(1)|_Z). 
\]
We then get  
\[
h^0(\P^n_k, \MO_{\P^n_k}(1) \otimes \mathcal I_Z) \geq 
 h^0(\P^n_k, \MO_{\P^n_k}(1)) - h^0(Z, \MO_{\P^n_k}(1)|_Z) 
 = n+1-d. 
\]
Set $W$  to be the intersection of all the hyperplanes $H$ corresponding to 
the elements of $H^0(\P^n_k, \MO_{\P^n_k}(1) \otimes \mathcal I_Z)$. 
Then 
\[
\dim W =  \dim \P^n_k - h^0(\P^n_k, \MO_{\P^n_k}(1) \otimes \mathcal I_Z) 
\leq n - (n+1-d) = d-1, 
\]
and thus $W$ is contained in a $(d-1)$-dimensional linear subvariety $V$.
\end{proof}

\begin{lem} \label{lem: d>6_rho>1}
  Assume that $k$ is a $C_1$-field. 
  Let $X$ be a geometrically integral regular del Pezzo surface  with $K_X^2 =6$ and $\rho(X)>1.$ 
    Then $X(k) \neq \emptyset$.
\end{lem}

\begin{proof}
By \cite[Lemma 4.24]{BFSZ}, $X$ does not admit two Mori fibre space structures onto curves.  
Therefore, $X$ admits a birational contraction $X \to Y$ onto a regular del Pezzo surface with $K_Y^2>K_X^2 =6$. 
By \autoref{c deg789}, we have $Y(k) \neq \emptyset$, 
which implies $X(k) \neq \emptyset$ by the Lang-Nishimura theorem. 
\end{proof}

\begin{prop}\label{p smooth case}
Assume that $k$ is a $C_1$-field. 
Let $X$ be a geometrically integral smooth del Pezzo surface. 
Then $X(k) \neq \emptyset$. 
\end{prop}

\begin{proof}
The assertion follows from  \cite[Theorem IV.6.8]{Kol96} or  \cite[the argument of Theorem 4.2]{Man66}. 
\end{proof}


From now on, we focus on the case when $X$ is not smooth. 

\begin{lem}\label{l deg6 p=2,3}
    Let $X$ be a geometrically integral regular del Pezzo surface 
    with $K_X^2 =6$. 
    If $X$ is not smooth, then $p \in \left\{2,3 \right\}$.
\end{lem}

\begin{proof}
{If $X$ is not geometrically normal, 
then we get $p \in \{2, 3\}$ by \autoref{thm: class_dP}. Hence} 
we may assume that $X$ is geometrically normal. 
    From 
    \cite[Theorem 3.3]{BT22} 
    and the classification of canonical singularities \cite[page 109]{kk-singbook}, $X_{\overline{k}}$ can only have 
    {$A_1$ and $A_2$} singularities.
    Note that regular twisted forms of $A_1$ singularities (resp.\,$A_2$ singularities) appear only in characteristic 2 (resp.\,3) by \cite[Theorem 6.1]{Sch08}. 
\end{proof}

Thus we are left to prove the existence of rational points for non-smooth del Pezzo surfaces of Picard rank 1 in characteristic $p=2,3$. We now divide the proof according to the characteristic.

\begin{prop}\label{p deg6 char2}
Assume that $k$ is a $C_1$-field of characteristic two. 
Let $X$ be a geometrically integral regular del Pezzo surface with $K_X^2=6$. 
Then $X(k) \neq \emptyset$. 
\end{prop}

\begin{proof}
If $X$ is not geometrically normal, then we are done by \autoref{cor: non-normal-rho2}.
Hence we may assume that $X$ is geometrically normal and $\rho(X)=1$ by \autoref{lem: d>6_rho>1}. 
By \cite[the proof of Proposition 6.6, especially the case (5)]{BT22}, $X$ has a purely inseparable point $P$ of degree two, i.e., $P$ 
is a closed point of $X$ such that $k(P)/k$ is a purely inseparable field extension satisfying $[k(P):k] =2$.

Let $X \subset \P^6_k$ be the anti-canonical embedding. 
Then there exists a line $L$ in $\P^6_k$ containing $P$ by \autoref{l 0 dim lin sub}. 
If $L \subset X$, then 
we get $-K_X \cdot L=1$ and $\deg K_L =-2$, 
and hence we may contract $L$, contradicting the hypothesis $\rho(X)=1$.
We then get $L \not\subset X$. 
Since $X$ is an intersection of quadrics by \autoref{p va deg3}(4), 
there is a quadric hypersurface $Q \subset \P^6_k$ such that $X \subset Q$ and $L\not\subset Q$. 
Then $L \cap Q$ is an effective divisor on $L=\P^1_k$ of degree two. 
Hence we get a scheme-theoretic equality $L \cap Q = P$, 
which implies $P \subset L \cap X \subset  L\cap Q =P$, and hence 
\[
L \cap X = P. 
\]

Since $L \cap X$ is an intersection of members of $|-K_X|$, 
the equality $L \cap X = P$ implies that the blow-up 
\[
\sigma  : Y \to X
\]
at $P$ coincides with the resolution of indeterminacies of the linear system corresponding to $H^0(X, \MO_X(-K_X) \otimes \mathfrak m_P)$. 
Hence $|-K_Y|$ is base point free. 
By $K_Y^2 = K_X^2 -2 =4>0$, $-K_Y$ is big.
Let $\tau \colon Y \to Z$ be the anti-canonical model of $Y$. 
Since $|-K_Z|$ is  base point free 
and $H^1(Z, \MO_Z)=0$ \cite[Proposition 4.8]{BM23}, 
 \autoref{p hypersurf}(4) is applicable, and hence 
$Z$ is a complete intersection of two quadric hypersurfaces. 

Thus we know that $Z(k) \neq \emptyset$ by \cite[Corollary at page 376]{Lan52}. 
Fix $Q \in Z(k)$. If $\tau : Y \to Z$ is isomorphic around $Q$, 
then we clearly have $Y(k) \neq \emptyset$. 
Hence we may assume that $Q$ lies in $\tau(\Ex(\tau))$. 
By $\rho(Y)=2$, $E=\Ex(\tau)$ is 
{irreducible, and hence} 
a Gorenstein curve. 
The exact sequence 
\[
0 \to \MO_Y(-E) \to \MO_Y \to \MO_E \to 0
\]
induces another one 
\[
\tau_*\MO_Y \to \tau_*\MO_E \to R^1\tau_*\MO_Y(-E) =0. 
\]
We then conclude $H^0(E, \MO_E) = k(Q) = k$, 
where $k(Q)$ denotes the residue field of $Q \in Z$. 
By adjunction, $-K_E$ is ample, and hence 
$E$ is a conic over the $C_1$-field $k=k(Q)$. 
We then get $Y(k) \supset E(k) \neq \emptyset$, which implies $X(k) \neq \emptyset$. 
\end{proof}

\begin{prop}\label{p deg6 char3}
Assume that $k$ is a $C_1$-field of characteristic three. 
Let $X$ be a geometrically integral regular del Pezzo surface with $K_X^2=6$. 
Then $X(k) \neq \emptyset$. 
\end{prop}

\begin{proof}
Let $X \subset \P^6_k$ be the anti-canonical embedding. 
We may assume that 
 $X$ is geometrically normal (\autoref{cor: non-normal-rho2}) 
and $\rho(X)=1$ (\autoref{lem: d>6_rho>1}). 
In particular, we get $\Pic X = \Z K_X$ as $K_X^2=6$ is not divisible by a square. 
It follows from \cite[Proposition 6.5]{BT22} that $X$ contains a purely inseparable point $P$ of degree $3$. 
By \autoref{l 0 dim lin sub}, there exists a plane $V \simeq \P^2_k$ in $\P^6_k$ such that $P \in V$. 
As $X \neq  V$, we have that $\dim (X \cap V) = 0$ or $\dim (X \cap V) = 1$.





\medskip 

Assume $\dim (X \cap V) =  1$. 
Fix a prime divisor $C \subset X \cap V$ on $X$. 
We have that 
\[
C \subset X \cap V  = \bigcap_{X \subset Q} (Q \cap V),  
\]
where $Q$ runs over all the quadric hypersurfaces  in $\P^6_k$  containing $X$. 
Since $X \cap V$ is one-dimensional, 
there exists a quadric hypersurface $Q \subset \P^6_k$ such that  
$Q \cap V \subsetneq V = \P^2_k$ is a conic. 
As $
C \subset X \cap V \subset Q \cap  V, $
the curve $C$ is contained in a conic $Q \cap  V$ on $V=\P^2_k$. 
We then get $\deg K_C <0$, which contradicts $C \subset X$ and $\Pic X = \Z K_X$, because 
the following holds for the integer $m>0$ satisfying $C \sim -mK_X$: 
\[
0>\deg K_C = (K_X +C) \cdot C  = (K_X +  (-mK_X)) \cdot C = 
(m-1)(-K_X) \cdot C \geq 0. 
\]

\medskip 

{In what follows, we assume $\dim (X \cap V)= 0$.} 
Let $\sigma: Y \to X$ be  the blowup at $P$. 
Then $K_Y^2 = K_X^2 -3 =3$. 
We now finish the proof by assuming that 
\begin{enumerate}
\item[($\star$)] $X(k) \neq \emptyset$, or $|-K_Y|$ is base point free and $-K_Y$ is big. 
\end{enumerate}
\noindent
By ($\star$), we may assume that $|-K_Y|$ is base point free and $-K_Y$ is big. 
Let $\tau \colon Y \to Z$ be the anticanonical model of $Y$, 
and hence $Z$ is a canonical del Pezzo surface with $K_Z^2 =3.$ 
Since $|-K_Y|$ is base point free, so is $|-K_Z|$. 
 By 
 $H^1(Z, \MO_Z)=0$ \cite[Proposition 4.8]{BM23}, 
  \autoref{p hypersurf}(3) is applicable, and hence  $Z$ is a cubic surface. 
Since the base field $k$ is $C_1$, 
we get $Z(k) \neq \emptyset$. Let $Q \in Z(k)$ and let $F=\Ex(\tau)$, which is a Gorenstein curve with $-K_F$ ample by adjunction. 
If $Q \neq \tau(\Ex(\tau))$, then we clearly have $Y(k) \neq \emptyset$.
If $Q =\tau(\Ex(\tau)),$ then we conclude that $F$ is a conic 
over $k$ as in the proof of \autoref{p deg6 char2}, and therefore $F(k) \neq \emptyset,$ concluding $X(k) \neq \emptyset.$ 

Let us show ($\star$). 
By $K_Y^2 =3>0$, it is enough to show that $X(k) \neq \emptyset $ or $|-K_Y|$ is base point free. 
Recall that 
\[
\dim (X \cap V) = 0, \qquad 
X = \bigcap_{Q \in \Sigma(X)} Q, 
\qquad \text{and}\qquad 
X \cap V = \bigcap_{Q \in \Sigma(X)} (Q \cap V), 
\]
where $\Sigma(X)$ denotes the set consisting of 
all the quadric hypersurfaces in $\P^6_k$ containing $X$. 
For $Q \in \Sigma(X)$, we set $C_Q := Q \cap V$. 
Note that  either $C_Q = V$ or $C_Q$ is a (possibly non-integral) conic on $V= \P^2_k$.  
By $\dim (X \cap V) =0$, there exist quadric hypersurfaces 
$Q_1, Q_2 \in \Sigma(X)$ such that, 
for each $i \in \{1, 2\}$, 
$C_{Q_i}$ is a conic on $V=\P^2_k$ and $\dim (C_{Q_1} \cap C_{Q_2})=0$. 
By $P \in C_{Q_1} \cap C_{Q_2}$, $\deg_k P = 3$, and 
$C_{Q_1} \cdot C_{Q_2} =4$ (which denotes the intersection number on $V=\P^2_k$), 
the scheme-theoretic intersection 
$C_{Q_1} \cap C_{Q_2}$ is reduced and 
we have $C_{Q_1}\cap C_{Q_2} = \{P, \wt{P}\}$ for some $k$-rational point $\wt{P}$. 
If $\wt{P} \in Q$ for every $Q \in \Sigma(X)$, then we get 
\[
\wt{P} \in \bigcap_{Q \in \Sigma(X)} Q = X, 
\]
which implies $X(k) \neq\emptyset$, and hence $(\star)$ holds. 
If $\wt{P} \not\in Q_3$ for some $Q_3 \in \Sigma(X)$, then we have scheme-theoretic inclusions
\[
\{ P\} \subset X \cap V  = \bigcap_{Q \in \Sigma(X)} C_Q 
\subset C_{Q_1} \cap C_{Q_2} \cap C_{Q_3} \subset \{ P\},  
\]
which implies a scheme-theoretic equality $X \cap V = \{ P\}$. 
Since $V$ is an intersection of some hyperplanes on $\P^6_k$ passing through $P$, 
the divisor 
$-K_Y = -\sigma^*K_X -E  \sim \sigma^*(\mathcal{O}_{\mathbb{P}^6_k}(1)|_X) -E$ is base point free for $E := \Ex(\sigma)$. 
Therefore, ($\star$) holds.
\qedhere

\end{proof}

\begin{prop}\label{p deg6 rat pt}
Assume that $k$ is a $C_1$-field. 
Let $X$ be a geometrically integral regular del Pezzo surface with $K_X^2=6$. 
Then $X(k) \neq \emptyset$. 
\end{prop}

\begin{proof}
By \autoref{p smooth case} and \autoref{l deg6 p=2,3}, we may assume that $p=2$ or $p=3$. 
If $p=2$ (resp. $p=3$), then the assertion follows from 
\autoref{p deg6 char2} (resp. \autoref{p deg6 char3}). 
\end{proof}

\subsection{Degree 5}




The aim of this section is to generalise the theorem of Enriques on the rationality of smooth del Pezzo surfaces of degree 5 \cite{Enr97} (proven also in \cite{SD72, SB92, Sko93}) to the geometrically integral regular case. 

\begin{theorem}\label{them: Enriques}
    Let $X$ be a {geometrically integral} regular del Pezzo surface with $K_X^2=5$.
    Then $X$ is rational.
\end{theorem}






    

{If $X$ is not geometrically normal, then it is a blowup of $\mathbb{P}^2_k$ by \autoref{thm: class_dP} and thus rational. 
Therefore, we may assume that 
$X$ is geometrically normal, and hence geometrically canonical. }


\begin{lem}\label{lem: dP_5_geom}
Let $k$ be an infinite field and let $X$ be a geometrically normal regular del Pezzo surface with  $K_X^2 =5$. 
Then $X$ is a scheme-theoretic intersection of $5$ quadrics $Q_1, ..., Q_5 \subset \P^5_k =V$ via its anti-canonical embedding. 
The same statement holds true for $X' \subset V'$, 
where $V'$  is a general hyperplane on $\P^5_k$ and $X' := X \cap V' \subset \P^4_k = V'$. 
\end{lem}

\begin{proof}
Note that $X$ is smooth outside finitely many points. 
Since $|-K_X|$ is very ample { \cite[Proposition 2.14(5)]{BT22}}, 
we can apply the Bertini theorem to conclude that a general member $X'$ of $|-K_X|$ is 
a smooth curve of degree $5$. 
By adjunction, $X'$ is of genus one, and hence $X'_{\overline k}$ is an elliptic curve. 
By \autoref{l va genus1}(3)(4), $X' \subset \P^4_{k}$ is projectively normal 
and an intersection of  quadrics.
Thus we get a short exact sequence  
\[
0 \to H^0(\P^4_k, \MO_{\P^4_k}(2) \otimes I_{X'}) \to 
 H^0(\P^4_k, \MO_{\P^4_k}(2)) \to H^0(X', \MO_{\P^4_k}(2)|_{X'}) \to 0, 
\]
which implies 
\[
h^0(\P^4_k, \MO_{\P^4_k}(2) \otimes I_{X'}) = 
h^0(\P^4_k, \MO_{\P^4_k}(2)) - h^0(X', \MO_{\P^4_k}(2)|_{X'}) = (6 \cdot 5)/2 - 10 = 5. 
\]
For a $k$-linear basis $ H^0(\P^4_k, \MO_{\P^4_k}(2) \otimes I_{X'}) = ks_1 \oplus  ...\oplus k s_5$ and the corresponding quadrics $Q'_1, ..., Q'_5 \subset \P^4_k$, we get 
\[
X' = \bigcap_{X' \subset Q} Q = Q'_1 \cap \cdots \cap Q'_5. 
\]
Hence the assertion of the lemma holds for $X'$.
By \cite[Lemma 2.10]{Isk77}, we conclude that $X$ is an intersection of 5 quadrics 
\[
X = Q_1 \cap \cdots \cap  Q_5, 
\]
where $Q_i \subset \P^5_k$ is a quadric hypersurface satisfying $Q_i \cap V'=Q_i'$. 
\end{proof}

\begin{nota}\label{n quintic}
\begin{enumerate}
\item Assume that $k$ is infinite. 
Let $X$ be a geometrically normal regular del Pezzo surface with  
$H^0(X, \MO_X)=k$ and $K_X^2 =5$. 
Recall that $X$ is a scheme-theoretic intersection of 
5 quadrics $Q_1, ..., Q_5 \subset \P^5_k$ (\autoref{lem: dP_5_geom}). 
Let $\sigma : W \to \mathbb{P}^5_k$ be  the blowup  along $X$, 
which coincides with the resolution of the indeterminacies of the linear system $\Lambda \subset |\MO_{\mathbb{P}^5_k}(2)|$ generated by $Q_1, ..., Q_5$. Let 
\[
\pi : W \to \P^4_k
\]
be the morphism 
induced by the base point free complete linear system $|2\sigma^*H-E|$, 
where $H := \MO_{\mathbb{P}^5_k}(1)$ and $E:=\Ex(\sigma)$. 
\item 
Take a general hyperplane $V' =: \P^4_k$ of $\P^5_k$. 
Set $X' \coloneqq X \cap V'$ and let $ \sigma' \colon W' = \Bl_{X'} V' \to V'$ be the blow-up of $V'$ along $X'$.
Note that $W' \subset W$ and $\sigma^*V'=W'$, 
because $X$ is not contained in $V'$. 
\end{enumerate}
\end{nota}



\begin{lem} \label{lem: blowup_P4}
We use Notation \ref{n quintic}. 
Then the following hold. 
\begin{enumerate}
    \item $X' = X \cap V' \subset V' = \P^4_k$ is a smooth curve of genus 1 and degree $5$. 
    \item The induced morphism $\pi|_{W'} : W' \to \P^4_k$ is birational. 
\end{enumerate}
\end{lem}

\begin{proof}
As $X$ is smooth outside a finite number of points, $X'$ is a smooth curve of genus 1 in $\mathbb{P}^4_k$, not contained in a plane, which is the intersection of 5 quadrics $Q_1', ..., Q_5'$ by \autoref{lem: dP_5_geom}. 
Then (1) holds.

Let us show (2). 
Set $E' := \Ex(\sigma')$. 
Note that the restriction $\pi|_{W'} : W' \to \P^4_k$ coincides with the morphism induced by the complete linear system $|2\sigma'^*H-E'|$ which resolves the sublinear system of quadrics containing $X'$. 
We now follow the proof of \cite[Theorem 2.2]{CK89}.
Since $\pi|_{W'} : W' \to \P^4_k$ is a generically finite morphism of degree $(2\sigma'^*H-E')^4,$ it is sufficient to verify $(2\sigma'^*H-E')^4=1$.
By  dimension reason, 
we have $(\sigma'^*H)^2 \cdot E'^2=(\sigma'^*H)^3 \cdot E'=0 $. Moreover, the projection formula shows that $(\sigma'^*H)^4=1$. 
Let $N_{X'/\mathbb{P}^4_k}$ be the normal bundle of $X'$ inside 
$V' = \mathbb{P}^4_k$.
By \cite[Example 3.3.4]{Ful98}, we have $E^{4-i}=(-1)^{3-i}c_1(\mathcal{O}_{\mathbb{P}(N_{X'/\mathbb{P}^4_k})}(1))^{3-i} \cap [E]$ for $i \in \left\{0, 1\right\}$. By \cite[Corollary 4.2.2]{Ful98} and the definition of the Segre class \cite[page 73]{Ful98},  we have $\sigma'^{*}H^i\cdot E'^{4-i}=(-1)^{3-i}H^i \cdot s_{1-i}(N_{X'/\mathbb{P}^4_k}) \cap [X']$, where $s_{j}$ denotes the $j$-th Segre class.
Then we deduce that $\sigma'^*H \cdot E'^3=H\cdot X'= 5$ and
$E'^4=-\deg_{X'}(s_1(N_{X'/\P^4_k}))=\deg_{X'}(N_{X'/\P^4_k})$.
By the conormal bundle short exact sequence   $0 \to N_{X'/\P^4_k}^\vee \to \Omega^1_{\P^4_k}|_{X'} \to \Omega^1_{X'} \to 0$, we deduce $E'^4=\deg(N_{X'/\P^4_k})=\deg_{X'}(-K_{\mathbb{P}^4_k})=25$.
Therefore we have
$$(2\sigma'^*H-E')^4 = 16-8 \sigma'^*H \cdot E'^3+E'^4=16-8 \cdot 5 +25=1,$$
concluding the proof.
\end{proof}


We say that a line $L$ in $\mathbb{P}^{{n}}_k$ is a \emph{secant} of a closed subvariety $Y \subset \mathbb{P}^n_k$ 
if the intersection $L \cap Y$ 
is of length at least 2. 
Let $X$ be  a geometrically normal regular del Pezzo surface $X$ 
with $H^0(X, \MO_X)=k$ and $K_X^2 =5$. 
For its anticanonical embedding $X \subset \mathbb{P}^5_k$, 
$X$ is an intersection of quadrics (\autoref{lem: dP_5_geom}). 
For a secant $L$ of $X$, 
 either $L \subset X$ or $L \cap X$ is a zero-dimensional closed subscheme of length 2.

\begin{lem}\label{lem: secant}
We use Notation \ref{n quintic}. 
Let $P$ be a general $k$-rational point of $\P^4_k$ and set $F := \pi^{-1}(P)$, 
which denotes the scheme-theoretic fibre. 
Then the following hold. 
\begin{enumerate}
\item 
$\pi : W \to \P^4_k$ is a contraction, i.e., $\pi_*\MO_W = \MO_{\P^4_k}$. 
\item 
$F$ is the proper transform of a secant  $L$ of $X$. 
Moreover, $F$ is smooth and geometrically integral.
\item 
For a general secant $L$ of $X$, its proper transform $L_W= 
{\sigma}_*^{-1}L$ on $W$ 
is equal to a scheme-theoretic fibre of $\pi$. 
\end{enumerate}
In particular, general fibres of $\pi$ are smooth and geometrically integral. 
\end{lem}

\begin{proof}
Let us show (1). 
By \autoref{lem: blowup_P4}, $\pi':= \pi|_{W'} \colon W' \to \P^4_k$ is birational and, as $\P^4_k$ is normal, $\mathcal{O}_{\P^4_k} = \pi'_{*}\MO_{W'}$. 
Let $W \to Z \to \P^4_k$ be the Stein factorisation of $\pi : W \to \P^4_k$. 
As the composition $W' \to W \to Z \to \mathbb{P}^4_k$ is birational, we conclude that $Z \to \mathbb{P}^4_k$ is a finite birational morphism between normal varieties, and hence it is an isomorphism. 
Thus (1) holds. 

Let us show (2). 
Since $\pi' : W' \to \P^4_k$ is birational, 
we get $W' \cdot F =1$. 

\begin{claim-}
$F$ is an integral scheme. 
\end{claim-}

\begin{claimproof}
    By (1),  $F$   is (geometrically) irreducible (cf. \cite[Lemma 2.2(1)]{Tan18b}). 
Hence we get $[F] = n [F_{\red}]$ as $1$-cycles for some integer $n>0$. 
We then get $1 =  W' \cdot F = n W' \cdot F_{\red}$, 
and thus $n=1$ and $[F] =[F_{\red}]$, i.e., $F$ is generically reduced. 
Recall that $F$ is Cohen-Macaulay, 
because the $k$-rational point  $P = \pi(F) \in \P^4_k$ is locally a complete intersection, and hence so is its fibre $F$. 
Therefore, $F$ is $R_0$ and $S_1$, and hence $F$ is an integral scheme. 
This completes the proof of Claim.
\end{claimproof}

\medskip

It holds that $1 = W' \cdot F =\sigma^*V' \cdot F = V' \cdot \sigma_*F$, 
and thus $L:=\sigma_*F$ is a line. 
Therefore, 
\[
0= \pi^*\MO_{\P^4_k}(1) \cdot F =(2\sigma^*H -E) \cdot F = 2H \cdot L - E \cdot F = 
2 - \sigma^{-1}(X) \cdot F.  
\]
This, together with $\sigma^{-1}(X) \cap F \simeq X \cap L$, implies that $L$ is a secant of $X$. 
In particular, $F ( \simeq L)$ is smooth and geometrically integral. 
Thus (2) holds.

Let us show (3). Let $L$ and $L_W$ be as in the statement. 
We have $L_W \xrightarrow{\sigma|_{L_W}, \simeq} L$ and 
\[
0 =2 - \sigma^{-1}(X) \cdot L_W= 2 H \cdot L -E \cdot L_W =
(2\sigma^*H -E) \cdot L_W = \pi^*\MO_{\P^4_k}(1) \cdot L_W. 
\]
Hence $L_W$ is contained in a fibre of $\pi :W \to \P^4_k$. 
For the $k$-rational point $P' := \pi(L_W)$ and 
its fibre $F' :=\pi^{-1}(P')$, 
  we get $[F'] = [L_W] + [\Gamma_1]+ \cdots +[\Gamma_r]$ as $1$-cycles for some curves $\Gamma_1, ..., \Gamma_r$. 
By $\rho(W)=2$, we get $\rho(W/\P^4_k)=1$, and hence $W'$ is $\pi$-ample as $W' \cdot F'=1$. 
Therefore, we get $r=0$, i.e., $[F'] =[L_W]$. 
By the same argument as in Claim, $F' =\pi^{-1}(P')$ is an integral scheme, and hence $L_W$ 
coincides with the scheme-theoretic fibre $\pi^{-1}(P')$. 
Thus (3) holds. 
\end{proof}

\begin{proof}[Proof of \autoref{them: Enriques}]
{If $X$ is not geometrically normal, then it is a blowup of $\mathbb{P}^2_k$ by \autoref{thm: class_dP} and hence rational. 
Therefore,} 
{we may assume that 
$X$ is geometrically normal, and thus geometrically canonical.}
If $k$ is a finite field, then $X$ has a rational point by 
\autoref{p smooth case}, 
and thus we conclude by \autoref{t rationality}.
From now on, we assume that $k$ is an infinite field.

We use \autoref{n quintic}. 
Let $y \in \mathbb{P}^4_k$ be a general $k$-rational point and let $L_y$ be the fibre of $\pi \colon W \to \mathbb{P}^4$ over $y$. Then the image $L \coloneqq \sigma(L_y)$ is a secant line of $X$ by \autoref{lem: secant}(2). 
Then $(X \cap L) \times_k \overline k$ consists of two distinct points $P$ and $Q$ (even in characteristic 2), because this property holds for a general secant defined over $\overline k$, 
which implies that the induced morphism $\pi|_E \colon E \to \P^4_k$ is a generically finite separable morphism of degree two. 
If $P$ (or $Q$) descends to a $k$-rational point on $X$, then we conclude from \autoref{t rationality}.

Otherwise, we set $R:=X \cap L = \Spec k'$, where $k'$ is a Galois extension of degree 2 of $k$. 
As the base field $k$ is infinite, 
we may assume that 
$P$ and $Q$ are general closed points 
{(i.e., given a non-empty open subset $U$ of $X_{\overline k} \times_{\overline k} X_{\overline k}$, we may assume $(P, Q) \in U$). 
Let $Y \to X$ be the blowup at $R$. 
We now finish the proof by assuming that 
\begin{enumerate}
\item[($\star$)] $Y$ is a regular del Pezzo surface. 
\end{enumerate}
}  
\noindent
Then $Y$ is a geometrically canonical regular del Pezzo surface with $\rho(Y)=2$, 
because its base change $Y_{\overline{k}}$ to the algebraic closure $\overline k$ is a blowup of a canonical {del Pezzo}  surface $X_{\overline k}$ 
at two general non-singular points $P$ and $Q$. 
As $Y_{\overline{k}}$ is a normal cubic surface containing two disjoint lines, interchanged by the Galois action, we can conclude that $Y$ is rational by \cite[Example 1.35]{KSC04}. 

{
In what follows, we prove ($\star$). 
Take the minimal resolution $\tau \colon Z \to X_{\overline{k}}$ 
and we set $F :=\Ex(\tau)$. 
Since $P$ and $Q$ are in general position, 
we get $P \not\in \tau(F)$ and $Q \not\in \tau(F)$. 
The condition that the blowup $\Bl_{P \amalg Q} X_{\overline{k}} = Y \times_k \overline k$ of $X_{\overline k}$ at the points $P$ and $Q$ remains a del Pezzo surface is equivalent to requiring that the blowup $\Bl_{P \amalg Q} Z$ is a weak del Pezzo surface with the same number of $(-2)$-curves as $Z$. The existence of such a pair of points is ensured by \cite[Remark 2.9]{MS24} and the fact that to obtain a $(-2)$-curve 
it is necessary to blowup a point on a $(-1)$-curve.}
\qedhere

\end{proof}

\begin{rem}
The last step of the above proof (i.e., the part using \cite[Example 1.35]{KSC04}) 
can be replaced by playing the two-ray game (Sarkisov link) on $Y$ as follows. 
Let $\tau : Y \to Z$ be the other contraction. 
By $K^2_{X} = 5$, it is easy to exclude the case when $\dim Z=1$ 
(i.e., the Sarkisov link of type I)   \cite[Proposition 4.32]{BFSZ}. 
Then $\tau : Y \to Z$ is a birational contraction onto a regular del Pezzo surface $Z$ with $K_Z^2=8$ by \cite[Proposition 4.35]{BFSZ}.
In particular, $Z$ has a closed point of degree $5$, and therefore 
it follows from \cite[Lemma 4.9, Proposition 4.10]{BFSZ} that $Z$ is a quadric surface in $\mathbb{P}^3_k$. 
By \cite[Theorem 18.5]{EKM08}, we conclude that $Z$ has a rational point,
and thus it is a rational surface.
\end{rem}

\subsection{Existence of rational points (over $C_1$-fields)}\label{ss rat pt}

{We are ready to prove \autoref{thm: C_1}, 
which is equivalent to the theorem below (\autoref{r non geom int}(4)).}

\begin{thm}\label{thm: existence_section}
Assume that $k$ be a $C_1$-field. 
Let $X$ be a geometrically integral regular del Pezzo surface. 
Then $X(k) \neq \emptyset$, i.e., $X$ has a $k$-rational point. 
\end{thm}

\begin{proof}
 If $K_X^2 \leq 4$, then the assertion follows from \cite[Lemma 6.3]{BT22}. 
For the remaining case $K_X^2 \geq 5$, 
we are done by \autoref{c deg789} ($K_X^2 \in \{ 7, 8, 9\}$), \autoref{p deg6 rat pt} 
($K_X^2 =6$), and 
\autoref{them: Enriques} ($K_X^2 =5$). 
\end{proof}



\begin{corollary}
    Assume that $k$ be a $C_1$-field. 
Let $X$ be a geometrically integral regular projective surface. 
If $\kappa(X)=-\infty$ and $H^1(X, \MO_X)=0$, then $X(k) \neq \emptyset$.
\end{corollary}

\begin{proof}
Assume that $\kappa(X) = -\infty$ and $H^1(X, \MO_X)=0$. 
It follows from \cite[Theorem 1.1]{Tan18} and \cite[Theorem 1.1]{Tan20} that $K_X$ is not pseudo-effective. 
Then, by running a $K_X$-MMP, 
we have a birational morphism $X \to Y$, where 
$Y$ is a regular projective surface which admits a Mori fibre space structure $\pi : Y \to B$. 
    By the Lang--Nishimura lemma (cf. \cite[Lemma 6.11]{BT22}), it is sufficient to prove $Y(k) \neq \emptyset$. 
    If $\dim B=0$, then $Y$ is a regular del Pezzo surface 
    and hence we get $Y(k) \neq \emptyset$ by  \autoref{thm: existence_section}.
    We may assume that $\dim B=1$. 
    Then $H^1(B,\mathcal{O}_B) \hookrightarrow H^1(X, \MO_X)=0$. Therefore, $B$ is a conic over a $C_1$-field $k$ and it has a rational point $P$. As the fibre $X_P$ of $\pi$ over $P$ is also a conic in $\P^2_k$, we conclude that $\emptyset \neq X_P(k) \subset X(k)$.
\end{proof}

We now prove the application to numerically trivial line bundles on del Pezzo fibrations. 

\begin{proof}[Proof of \autoref{cor: torsion_index}]
    We follow the proof of \cite[Theorem 8.2]{BT22}. 
    As $X_{K(B)}$ is a regular del Pezzo with $\rho(X_{K(B)})=1$ over a function field of a curve, 
    $X_{K(B)}$ is geometrically normal by \cite[Theorem 14.1]{FS20}, and 
    hence $H^1(X_{K(B)}, \MO_{X_{K(B)}})=0$ and $L|_{X_{K(B)}} \sim 0$ (\autoref{lem: Kod_vanishing}). 
    This implies that
    $L \sim \sum l_i D_i$, where $l_i \in \mathbb{Z}$ and $D_i$ are prime divisors such that $\pi(D_i)$ is a closed point $b_i$.
    Since $\rho(X/B) = 1$ and $X$ is $\mathbb{Q}$-factorial, all the fibres of $\pi$ are irreducible. Hence we can write $\pi^*(b_i) = n_iD_i$ for some $n_i \in \mathbb{Z}_{>0}$. To conclude, it is sufficient to show that $l_i \in n_i \Z$. 
    As $l_iD_i$ is Cartier, it is enough to show that $n_i$ is the Cartier index of $D_i$, i.e., the smallest positive integer $m$ such that $mD_i$ is Cartier. 
    
    For this, we take a section $\Gamma \subset X$ of $\pi$ given by \autoref{main dP fib}.
    The equation 
    $\Gamma \cdot \pi^*(b_i) = \Gamma \cdot (n_i D_i)=1$ implies that the Cartier index of $D_i$ coincides with $n_i$, concluding the proof.
\end{proof}

\subsection{Unirationality }

We discuss unirationality of del Pezzo surfaces of degree 4.




\begin{theorem} \label{thm: unirat_deg_4}
    Let $X$ be a geometrically integral regular del Pezzo surface with $K_X^2 =4$ and $X(k) \neq \emptyset$. 
    Assume that $X$ is not unirational. 
    Then $p=2$, $X$ is primitive, and $X$ is not geometrically normal. 
\end{theorem}


\begin{proof}
We divide the proof into several steps.

\setcounter{step}{0}

\begin{step}\label{s1 unirat_deg_4}
$X$ is primitive. In particular, $\rho(X) \leq 2$. 
\end{step}

\begin{s1proof}
Suppose that $X$ is imprimitive. 
By \autoref{d primitive}, there is a birational morphism $f : X \to X'$ 
to a geometrically integral regular del Pezzo surface $X'$ with $4=K_X^2 <K_{X'}^2$. 
Then $X'$ is rational (\autoref{t rationality}), which is absurd. 
This completes the proof of Step \ref{s1 unirat_deg_4}.
\end{s1proof}

\medskip

It is enough to show that $X$ is not geometrically normal (\autoref{thm: class_dP}). 
Suppose that $X$ is geometrically normal. 
Let us derive a contradiction.




\begin{step}\label{s2 unirat_deg_4}
There exists $P \in X(k)$ such that $-K_Y$ is ample, where
$Y$ is the surface obtained by the blowup $\sigma: Y \to X$ of the closed point $P$. 
\end{step}


\begin{s2proof}
Pick $P \in X(k)$ and let $\sigma: Y \to X$  
be the blowup at $P$.

Assume  $\rho(X)=1$. 
As $H^1(X, \MO_X)=0$, we have that $\Pic X =\Z K_X$ by  \autoref{lem: Kod_vanishing} and \cite[Corollary 4.13]{BFSZ}. 
Then $-K_Y$ is ample by \autoref{l rho=1}.

In what follows, assume $\rho(X)=2$. 
Since there is no non-trivial birational contraction from $X$ (Step \ref{s1 unirat_deg_4}), 
we have two Mori fibre spaces $\pi_1 : X \to B_1$ and $\pi_2 : X \to B_2$ with $\dim B_1 = \dim B_2 =1$. 
For each $i \in \{1, 2\}$, 
set $Q_i := \pi(P) \in B_i(k)$ and $F_i :=\pi_i^{-1}(Q_i)$. 
We have $F_1 \cdot F_2 = 2$ and 
$-K_X \sim F_1+F_2$ (cf. the proof of Step \ref{s2 2 conic bdls} of \autoref{p 2 conic bdls}). 
Since $X$ is smooth at the $k$-rational point $P$, it holds that 
\[
2 = F_1 \cdot F_2 \geq (\mult_P F_1) \cdot (\mult_P F_2). 
\]
Hence 
$\mult_P F_1=1$ or $\mult_P F_2=1$, i.e., 
one of $F_1$ and $F_2$ is smooth at $P$. 
By symmetry, we may assume that $F_1$ is smooth at $P$. 
We then get $F_1 \simeq \P^1_k$, because 
we have $F_1(k) \neq \emptyset$ and $F_1$ is a smooth conic on $\P^2_k$ \cite[Proposition 2.18]{BT22}. 
Since $X$ is geometrically normal, 
$\pi_2 : X \to B_2$ is generically smooth (\autoref{p always wild}). 
Replacing $P$ by a general $k$-rational point on $F_1(k) (\simeq \P^1_k(k))$, 
we may assume that also $F_2$ is smooth, 
and hence $F_2 \simeq \P^1_k$.


We then have $-K_Y \sim  \sigma^*(F_1+F_2)-E \sim  F_{1, Y}+F_{2, Y}+E$, 
where $E := \Ex(\sigma)$ and each $F_{i, Y}$ denotes the proper transform of $F_i$ on $Y$. 
Let $C$ be a curve on $Y$. 
By $K_Y^2 =K_X^2 -1 =3>0$, it is enough to show $-K_Y \cdot C>0$.  
If $C \in \{E, F_{1, Y}, F_{2, Y}\}$, then $-K_Y \cdot C = 1$. 
We may assume $C \not\in \{E, F_{1, Y}, F_{2, Y}\}$. 
In particular, $\sigma(C)$ is still a prime divisor. 
Since $F_1 +F_2 (\sim -K_X)$ is ample, 
$\sigma(C)$ intersects $F_1 \cup F_2$. 
Then $C$ intersects $F_{1, Y} \cup F_{2, Y} \cup E$. 
By $C \not\in \{E, F_{1, Y}, F_{2, Y}\}$, we get 
$-K_Y \cdot C=(F_{1, Y}+F_{2, Y}+E) \cdot C >0$. 
This completes the proof of Step \ref{s2 unirat_deg_4}.
\end{s2proof}

\medskip

In what follows, we use the same notation as in the statement of Step \ref{s2 unirat_deg_4}. 
Set $E := \Ex(\sigma)$.

\begin{step}\label{s3 unirat_deg_4}
$Y$ is (isomorphic to) a cubic surface in $\P^3_k$ and $E$ is a line in $\P^3_k$. 
Moreover, the projection $\rho : Y \to \P^1_k$ from the line $E$ is the morphism induced 
by the base point free linear system $|-K_Y-E|$. 
\end{step}

\begin{s3proof}
We have $K_Y^2 =3$  and $-K_Y$ is ample. 
Hence $Y$ is (isomorphic to) a cubic surface in $\P^3_k$ by 
\autoref{thm: class_dP} and \cite[Theorem 2.15(3)]{BT22}. 
It follows from the adjunction formula that $\MO_{\P^3}(1)|_Y \simeq \MO_Y(-K_Y)$. 
We then get $\MO_{\P^3}(1) \cdot E = -K_Y \cdot E =1$, 
and hence $E$ is a line. 
Then the projection from the line $E$ gives a conic bundle structure $Y \to \mathbb{P}^1_k$ given by the linear system $|-K_Y-E|$. 
If $F$ is the residual conic of a plane $H$ containing $E$  (i.e., $H \cap Y = E \cup F$), then $(-K_Y-E) \cdot F=2-E \cdot F=0$ and $(-K_Y-E) \cdot E=1+1=2.$ 
Alternatively, the base point freeness of $|-K_Y-E|$ can be confirmed 
by taking the blowup of $\P^3_k$ along $E$. 
This completes the proof of Step \ref{s3 unirat_deg_4}.
\end{s3proof}

\begin{step}\label{s4 unirat_deg_4}
$X$ is unirational. 
\end{step}

\begin{s4proof}
The restriction $\rho|_E \colon E =\mathbb{P}^1_k \to \mathbb{P}^1_k=:A$ is a finite morphism of degree 2 and let $Z := Y \times_A E \to E  = \mathbb{P}^1_k$ be the base change of $\rho : Y \to \P^1_k=A$. 
Since $\rho : Y \to \P^1_k =A$ is a generically smooth conic bundle 
{(\autoref{p always wild})}, 
so is its base change $Z \to E$. 
Then, for $K := K(E)$, 
the generic fibre $Z_K$ is a smooth conic with a $K$-rational point, 
and hence $Z_K \simeq \P^1_K$. 
By $K =K(E) = K(\P^1_k) \simeq k(t)$, we see that $Z$ is rational, i.e., 
$K(Z) =K(Z_K) \simeq k(s, t)$. 
Therefore, its image $Y$ is unirational and hence so is $X$. 
This completes the proof of Step \ref{s4 unirat_deg_4}.
\end{s4proof}

\medskip

Step \ref{s4 unirat_deg_4} completes the proof of \autoref{thm: unirat_deg_4}.
\qedhere






\end{proof}







\begin{remark} \label{rem: quartic}
Note that there actually exists a non-unirational geometrically integral regular del Pezzo surface $X$ satisfying $X(k) \neq \emptyset$ and $K_X^2 =4$. 
    Consider the example \cite[Section 4]{OS22} of the cubic surface in characteristic 2:
    $$ Y=\left\{y^3_1 + t_1x^2_1y_1 + y^3_2 + t_2x^2_2y_2 = 0 \right\} \subset \mathbb{P}^3_{\mathbb{F}_2(t_1, t_2)}.$$
    Then $Y$ is not unirational by \cite[Theorem 4.4]{OS22} and it contains the line $L=\left\{y_1=y_2=0 \right\}$.
    Therefore, we can contract $L$ via a birational contraction $ \pi \colon Y \to X$, where $X$ is a regular del Pezzo surface of degree 4, which is not unirational.
\end{remark}

The case of del Pezzo surfaces of degree 3 has been discussed in \cite{Kol02} and \cite{OS22}.
Based on their work, we give a sufficient condition for unirationality.

\begin{theorem} \label{thm: unirat_cubics}
	Let $X\subset \P^3_k$ be a geometrically integral regular cubic surface with $X(k) \neq \emptyset$. 
 Then $X$ is unirational if the following conditions are satisfied: 
    \begin{enumerate}
        \item for some smooth $k$-rational point $x$ of $X$, the projection $\pi_x \colon X \dashrightarrow \mathbb{P}^2_k$ from $x$ is separable;
        \item for a general $k^{\sep}$-rational point $x \in X_{k^{\sep}}$, the curve $C_x =T_x X \cap X$ is a geometrically integral singular cubic curve, where $T_x X$ is the hyperplane tangent to $X_{k^{\sep}}$ at $x$.
    \end{enumerate}
    Moreover, conditions (1) and (2) always hold for $p \geq 5$, whilst condition (1) always holds for $p \geq 3$.
\end{theorem}

\begin{proof}
    We can suppose that $k$ is infinite by \cite[Theorem 1]{Kol02}.
    Let $x \in X(k)$ be a rational point and consider the projection $\pi_x \colon X \dashrightarrow \mathbb{P}^2_k$ from $x$, which is a rational map of degree 2.
    By (1), we may assume that $\pi_x$ is separable (note this is automatic if $p > 2$).
    
    Let $y \in \mathbb{P}^2_k$ be a general $k$-rational point. 
    Then 
    the inverse image 
    $\pi_x^{-1}(y)$ is either a disjoint union of two $k$-rational points  or a single closed point whose residue field is a separable quadratic extension of $k$.
    Let $L$ be the residue field of a closed point in  $\pi_x^{-1}(y)$.  
    As $L/k$ is separable, the base change $X_L := X \times_k L$ is still a geometrically integral regular cubic surface. 
    For  the $L$-rational point $y_L:= y\times_k L \in \P^2_L$ 
    lying over $y$,  $\pi^{-1}_x(y_L)$ is the union of $L$-rational points $P$ and $Q$. 
    Denote by $T_P X_L$ the tangent plane to the cubic $X_L$ at $P$ in $\mathbb{P}^3_L$.
    Set  $C_P :=T_P X_L \cap X_L$ and $C_Q:=T_Q X_L \cap X_L$. 
        As $P$ { (resp. $Q$) is} general,     
        $C_P$ { (resp. $C_Q$)  is a}   curve with a double point at $P$ (resp. $Q$) by (2). 
    Since $C_P$ and $C_Q$ are geometrically integral by hypothesis (2), 
    we see that $C_P$ is birational to $\mathbb{P}^1_L$ by using the projection from $P$ inside the plane $T_P X\simeq \mathbb{P}^2_L$.
    We thus have Galois--conjugate rational maps $\mathbb{P}^1_L \dashrightarrow C_P$ and $ \mathbb{P}^1_L \dashrightarrow C_Q$.

    We can now apply 
    \cite[8(Second unirationality construction) in page 469]{Kol02} to construct a rational map $\Psi: R_{L/k} \mathbb{P}^{1}_L \dashrightarrow  X$,  where $R_{L/k}$ denotes the Weil restriction of scalars.
    By hypothesis (1) and \cite[Lemma 15]{Kol02}, we conclude that $\Psi$ is a dominant rational map, concluding the proof as 
    {$R_{L/k} \mathbb{P}^{1}_L$ is rational} by \cite[Definition 2.1]{Kol02}. 
\end{proof}

\begin{remark}
The assumptions (1) and (2) can not be dropped from \autoref{thm: unirat_cubics}. 
Let $\F$ be an algebraically closed field of characteristic $p \in \{2, 3\}$. 
\begin{itemize}
\item If $p=2$, then we set 
\[
Y :=\left\{y^3_1 + t_1x^2_1y_1 + y^3_2 + t_2x^2_2y_2 = 0 \right\} \subset \mathbb{P}^3_k
\]
for $k:= \F(t_1, t_2)$. 
\item If $p=3$, 
then we set 
\[
Y := \left\{ y^3-yz^2+t_1x_1^{{3}}+t_2x_2^{{3}}=0 \right\} \subset \mathbb{P}^3_{k}
\]
for $k:= \F(t_1, t_2)$. 
\end{itemize}
Set $X := Y_{k^{\sep}}$ for each $p \in \{2, 3\}$.  
Then $X({k^{\sep}})$ is dense in $X$ and 
$X$ is not unirational by  
\cite[Theorem 4.4]{OS22} ($p=2$) and 
\cite[Theorem 2.7]{OS22} ($p=3$). 
\end{remark}

\begin{rem}\label{r non geom int}
Some results established in this article are extended to the geometrically non-integral case. 
\begin{enumerate}
\item 
$X$ is rational if $k$ is a field of characteristic $p>0$, 
$X$ is a regular del Pezzo surface $X$ over $k$, 
$H^0(X, \MO_X)=k$, and $K_X^2 \in \{5, 7\}$. 
Indeed, we get $\epsilon(X/k)=0$ by the list of \cite[Theorem 4.6]{Tan19},  which implies that $X$ is geometrically integral by \cite[Theorem 7.3]{Tan21}. 
Then we may apply \autoref{them: Enriques}  $(K_X^2 =5)$ and \autoref{p deg7} $(K_X^2 =7)$. 
\item 
Let $k$ be a field and let $X$ be a regular variety over $k$. 
If $P$ is a $k$-rational point of $X$, then $X$ is smooth around $P$ (cf. \cite[Proposition 2.13]{Tan21}), 
and hence $X$ is automatically geometrically integral when 
$X$ is projective over $k$ and $H^0(X, \MO_X)=k$. 
In particular, the geometrically integral assumption in \autoref{t rationality} 
can be replaced by $H^0(X, \MO_X)=k$. 
\item 
{Let $k$ be a field and let $X$ be a projective normal variety over $k$. 
If $X$ is geometrically integral, then we get $H^0(X, \MO_X)=k$.} 
\item 
 Let $k$ be a $C_1$-field and let $X$ be a projective normal variety over $k$. 
 {By \cite[Lemma 6.1]{BT22} and \cite[Theorem at the first page]{Sch10}, 
 $H^0(X, \MO_X)=k$ if and only if 
 $X$ is geometrically integral}. 
In particular, 
the geometrically integral assumption in \autoref{thm: existence_section} 
can be replaced by $H^0(X, \MO_X)=k$. 
\end{enumerate}
\end{rem}

\bibliographystyle{amsalpha}
\bibliography{refs}
	
\end{document}